\documentclass[11pt]{article}
\usepackage{amsfonts}
\usepackage{bbm}
\usepackage{}
\usepackage{latexsym}
\usepackage{color}
\usepackage[toc,page,title,titletoc,header]{appendix}
\usepackage{amssymb, graphicx, amsmath}
\usepackage{amsthm}
\usepackage{subfigure}
\usepackage{graphicx}
\usepackage{url}
\usepackage{CJK,graphicx}
\usepackage{algorithm}
\usepackage{multirow}
\usepackage{algpseudocode}
\usepackage{subfigure}
\usepackage{float}
\usepackage{hyperref}
\usepackage{amsfonts,amsmath,amsthm,amsfonts,amssymb}
\usepackage{helvet}
\usepackage{courier}
\usepackage{graphics,epsfig,graphicx,subfigure}
\usepackage[table]{xcolor}
\usepackage{helvet,epsfig,graphicx,subfigure}
\usepackage[table]{xcolor}
\usepackage{courier}
\usepackage{booktabs}
\usepackage{multirow}

\usepackage{pgf}
\usepackage{tikz}
\usetikzlibrary{arrows, decorations.pathmorphing, backgrounds, positioning, fit, petri, automata}

\newcommand*{\numlambda}{1}
\newcommand*{\numa}{2}
\tikzset{elegant/.style={smooth,thick,samples=50,line width=1.2pt}}
\tikzset{eaxis/.style={->,>=stealth}}

\definecolor{myblue}{rgb}{0,0,0.5}
\definecolor{mygreen}{rgb}{0,0.5,0}
\definecolor{myred}{rgb}{0.5,0,0}

\def \x{{\mbox{\boldmath $x$}}}

\newcommand{\bw}{{\mbox{\boldmath $w$}}}

\newcommand{\bbR}{{\mathbb{R}}}

\newtheorem{theorem}{Theorem}[section]
\newtheorem{lemma}[theorem]{Lemma}
\newtheorem{proposition}[theorem]{Proposition}

\newtheorem{definition}[theorem]{Definition}
\newtheorem{assumption}[theorem]{Assumption}
\newtheorem{remark}[theorem]{Remark}
\newtheorem{example}[theorem]{Example}

\def \[{\begin{equation}}
\def \]{\end{equation}}

\title{Perturbation techniques for convergence analysis of proximal gradient method and other first-order algorithms via variational analysis}

\date{First version: October 2017 / Second version: July 2018}

\author{Xiangfeng Wang\thanks{Shanghai Key Lab for Trustworthy Computing, School of Computer Science and Software Engineering, East China Normal University, China. The research of this author was supported by NSFC No. 11501210. Email: xfwang@sei.ecnu.edu.cn.}
\and Jane J. Ye \thanks{Department of Mathematics and Statistics, University of Victoria, Canada. The research of this author was partially
supported by NSERC. Email: janeye@uvic.ca.}
\and Xiaoming Yuan\thanks{Department of Mathematics,  The University of Hong Kong, Hong Kong, China. This author was supported by the General Research Fund from Hong Kong Research Grants Council: 12302318. Email: xmyuan@hku.hk}
\and Shangzhi Zeng\thanks{Department of Mathematics,  The University of Hong Kong, Hong Kong, China. Email: zengsz@connect.hku.hk}
\and Jin Zhang\thanks{Corresponding author: Department of Mathematics, Hong Kong Baptist University, Hong Kong.
           HKBU Institute of Research and Continuing Education, Shenzhen, China.
             Email: zhangjin198637@gmail.com}
}

\begin{document}
\maketitle

\begin{abstract}
We develop new perturbation techniques for conducting convergence analysis of various first-order algorithms for a class of nonsmooth optimization problems. We consider the iteration scheme of an algorithm to construct a perturbed stationary point set-valued map, and define the perturbing parameter by the difference of two consecutive iterates. Then, we show that the calmness condition of the induced set-valued map, together with a local version of the proper separation of stationary value condition, is a sufficient condition to ensure the linear convergence of the algorithm. The equivalence of the calmness condition to the one for the canonically perturbed stationary point set-valued map is proved, and this equivalence allows us to derive some sufficient conditions for calmness by using some recent developments in variational analysis. These sufficient conditions are different from existing results (especially, those error-bound-based ones) in that they can be easily verified for many concrete application models. Our analysis is focused on the fundamental proximal gradient (PG) method, and it enables us to show that any accumulation of the sequence generated by the PG method must be a stationary point in terms of the proximal subdifferential, instead of the limiting subdifferential. This result finds the surprising fact that the solution quality found by the PG method is in general superior. Our analysis also leads to some improvement for the linear convergence results of the PG method in the convex case. The new perturbation technique can be conveniently used to derive linear rate convergence of a number of other first-order methods including the well-known alternating direction method of multipliers and primal-dual hybrid gradient method, under mild assumptions.
\end{abstract}


\section{Introduction}\setcounter{equation}{0}

We illustrate our technique for
 the following (possibly) nonconvex and nonsmooth optimization problem:
\begin{equation}\label{Basic_Problem}
\min_{x\in \bbR^n}\ F(x) := f(x) + g(x),
\end{equation}
where $f:{\bbR^n} \rightarrow (-\infty, \infty]$ is a  proper lower semi-continuous (lsc) function that is  smooth  in its domain ${\rm dom}f:=\{x~|~ f(x)<\infty\}$    and $g:{\bbR^n} \rightarrow (-\infty, \infty]$ is  a proper lsc and possibly nonsmooth function.

Various data fitting problems in areas such as machine learning, signal processing, and statistics can be formulated in the form of  (\ref{Basic_Problem}), where $f$ is a loss function measuring the deviation of observations from a solution point and $g$ is a regularizer intended to induce certain structure in {the solution point}.
With the advent of big data era,  the problem instances are typically of large scale; and in recent years first-order methods such as the proximal gradient (PG) method originated from \cite{passty1979} (see also \cite{polyak1987introduction}),  block coordinate descent-type methods and their extended accelerated versions are popularly used to solve problem (\ref{Basic_Problem}).
In this paper, we concentrate on the study of the PG method.
For solving problem (\ref{Basic_Problem}), recall that the iterative scheme of the PG method is
\begin{equation}\label{PG_method_basic}
x^{k+1} \in  \hbox{Prox}_{g}^{\gamma}\left( x^k - \gamma \nabla f(x^k) \right),
\end{equation} where
$\gamma>0$ represents the  step-size and the proximal operator associated with  $g$ is defined as
\begin{equation}\label{Basic_Proximal_Problem}
\hbox{Prox}_{g}^{\gamma} \left( a\right) := \arg\min_{x\in {\bbR^n}} \left\{ g(x) + \frac{1}{2\gamma}\left\| x - a \right\|^2 \right\}.
\end{equation}
When $g$ is an indicator function of a closed convex set, the PG method reduces to the projected gradient method (see, e.g., \cite{polyak1987introduction}); when $f\equiv0$, it reduces to the proximal point algorithm (PPA) \cite{passty1979}; and when $g\equiv0$ it reduces to the standard gradient descent method (see, e.g., \cite{Cauchy1847}).

Throughout this paper, unless otherwise stated, we assume that the following assumptions hold.
 \begin{assumption}[Standing Assumption I]\label{assum_sc_h}
	{\rm (i)}  $f$ is smooth with $L$-Lipschitz gradient with $L>0$ which means that  $f$ is smooth on ${\rm dom}f $ which is assumed to be   open and  $\nabla f(x)$ is Lipschitz continuous on a closed set $C\supseteq {\rm dom}f\cap {\rm dom } g$ with constant $L$.
{\rm (ii)} $g(x)$ is continuous in  ${\rm dom } g$.
\end{assumption}
\begin{assumption}[Standing Assumption II]\label{assum_sc_gprox} $F(x)\geq F_{\rm min}$ for all $x$ in $\bbR^n$.
 $g$ is  prox-bounded which means that  the proximal operator $\hbox{Prox}_{g}^{\gamma}(\cdot)$ is well-defined when $\gamma$ is selected as $0< \gamma <\gamma_g$ for certain $\gamma_g >0$.
\end{assumption}
In addition, we assume that a global optimal solution of the optimization problem (\ref{Basic_Proximal_Problem}) is easy to calculate for any $a\in {\bbR^n}$ along with certain well chosen $\gamma > 0$.
This assumption can be satisfied by many important applications because for many popular regularizers like the $l_1$-norm and group sparse penalty, the subproblems (\ref{Basic_Proximal_Problem})  all have closed form solutions.
Even for nonconvex penalties such as the  SCAD and MCP to be discussed below,   closed-form solutions may be found if $\gamma$ is chosen appropriately.

{{It is known that various first-order methods for the convex case of problem (\ref{Basic_Problem}) converge at the $O\left(\frac{1}{k}\right)$ or $O\left(\frac{1}{k^2}\right)$ sublinear rates, where $k\geq 1$ is the number of iterations; see, e.g., \cite{BeckTeboulle2009,Nesterov2013introductory,TaoBoleyZhang2016,Tseng2010approximation}.}} However, for problem (\ref{Basic_Problem}) with certain specific structure, it has been observed numerically that many of them converge at a faster rate than that suggested by the theory; see, e.g., \cite{xiao2013proximal}. In particular, when $f$ is strongly convex and $g$ is convex, \cite{Odonoghue2015adaptive, Schmidt2011convergence} has proved the global linear convergence rate of the PG method with respect to the sequence of objective function values.

Many application problems have nonconvex data-fidelity objectives $f(x)$.
For instance, { the nonconvex neural network based loss function has been very popular} in the deep learning literature \cite{Goodfellow2016deep,Lecun2015deep}.
To see this, we present {{a simple neural network (NN)}} model for illustration. For a given dataset $\left\{ a_i, b_i \right\}_{i=1}^m$ with $a_i\in \mathbb{R}^{n}$ and $b_i\in \mathbb{R}$, for simplicity
we assume the input layer has $n$ nodes and the output layer has only one node, while one hidden layer with $p$ nodes is introduced.
The whole neural network is fully connected. {{
We denote by  $w_{jk}$  the  weight from node $j$ in the hidden layer to node $k$ in the input layer and for the input layer and $u_j$ the  weight from node  in the output layer to node $j$ in the hidden layer.
In both  the hidden layer and output layer, the sigmoid activation function $\sigma(a) = \frac{1}{1+e^{-a}}$}} is introduced into this model and the $l_2$ loss is applied in the output layer.
{{As a result, the mathematical formulation for this NN model can be written as
$$
\min_{ \bw_{1}, \cdots, \bw_p, {\bf u} } \ \overbrace{\frac{1}{2} \sum_{i=1}^m \left\| \sigma \left( \sum_{j=1}^{p}u_j \sigma\left(  \bw_{j}^T a_i\right) -b_i\right) \right\|^2}^{f(\bw_{1},\cdots,\bw_{p},{\bf u})} + g(\bw_{1},\cdots,\bw_{p},{\bf u}),
$$where $\bw_j:=(w_{j1},\dots, w_{jn})$, ${\bf u}:=(u_1,\dots, u_k\}$ and $g$ denotes the regularizer \cite{Reed1993Pruning}. Hence in general the function $f$ in the NN model is nonconvex which fulfills Assumption \ref{assum_sc_h}. }}

Moreover allowing the regularization term $g$ to be nonconvex further broadens the  range of applicability  of problem (\ref{Basic_Problem}). Indeed, nonconvex regularizers such as the smoothly clipped absolute deviation (SCAD) (\cite{Fan2001SCAD}) and the minimax concave penalty (MCP)  (\cite{Zhang2010MCP}) are known to induce ``better" sparser solutions in the sense that they induce nearly unbiased estimates which under some conditions are provably consistent, and the resulting estimator is continuous in the data which reduces instability in model prediction.  Hence studying the PG method with nonconvex $f$ and $g$ is an urgent task.

So far most of results  about the PG methods in the literature assume that the function $g$ is convex. In this case,  the proximal operator defined as in (\ref{Basic_Proximal_Problem})  is a single-valued map and we can define   the set of stationary points ${\cal X}$  as follows:
$$\bar{x}\in {\cal{X}} \Longleftrightarrow 0\in \nabla f(\bar x)+\partial g(\bar x) \Longleftrightarrow \bar  x=\hbox{Prox}_{g}^{\gamma}\left( \bar x - \gamma \nabla f(\bar x) \right),$$
where $\partial g (x)$ represents the subdifferential in the sense of convex analysis.
Following \cite[Assumption 2a]{Tseng2009coordinate}, we say  the Luo-Tseng error bound holds if for any $\xi \ge \inf_{x\in {\bbR^n}} F(x)$, there exist constant $\kappa > 0$ and $\epsilon > 0$, such that
\begin{equation}\label{Luo-Tseng-error-bound}
{\rm{dist}} (x,{\cal{X}}) \le\kappa\left\| x-\hbox{Prox}^{\gamma}_{g} \left( x - \gamma \nabla f(x) \right)  \right\|,\ \hbox{whenever}\ \  F(x)\le \xi,\ \left\|x- {\rm{Prox}}^{\gamma}_{g} \left( x - \gamma \nabla f(x) \right)  \right\| \le \epsilon.
\end{equation}
In  \cite{LuoTseng1993}, Luo and Tseng introduced a general framework of using the Luo-Tseng error bound
together with the assumption of the proper separation of isocost surfaces of $F$ on ${\cal X}$, i.e.,
\begin{equation}
\exists \delta>0 \mbox{ such that } x \in {\cal X}, y \in {\cal X},\ F(x)\not =F(y)\quad \Longrightarrow\quad \left\| x - y \right\|\ge \delta \label{isocost}
\end{equation}
 to prove the linear convergence of feasible descent methods which include the PG method for problem (\ref{Basic_Problem}) with $g(x)$ being an indicator function of a  closed convex set. Since the Luo-Tseng error bound is an abstract condition involving the proximal residue $\left\| x-\hbox{Prox}^{\gamma}_{g} \left( x - \gamma \nabla f(x) \right)  \right\|$ and the distance to the set of stationary points, it is not directly verifiable. Hence an important task is to find verifiable sufficient conditions based on functions $f$ and $g$ under which the Luo-Tseng error bound holds.
Unfortunately there are very few concrete cases where the  Luo-Tseng error bound condition holds. Nerveless, it is known that the Luo-Tseng error bound condition holds under one of the conditions (C1)-(C4), see e.g.  \cite{Tseng2009coordinate}.
\begin{assumption}[Structured Assumption]\label{assum_struc}
$f(x) = h(Ax)+\langle q,x\rangle$
 where $A$ is some given $m\times n$ matrix, $q$ is some given vector in $\bbR^n$, and $h:\bbR^m\rightarrow (-\infty,\infty]$ is  closed, proper, and convex with the properties that
   $h$ is continuously differentiable on dom$h$, assumed to be open and
 $h$ is strongly convex on any compact convex subset of dom$h$.
\end{assumption}

	\begin{itemize}
\item[{(C1)}] $f$ is strongly convex, $\nabla f$ is Lipschitz continuous, and $g$ is closed, proper, and convex.
\item[{(C2)}] $f$ satisfies Assumption \ref{assum_struc}, and $g$ has a polyhedral epigraph.
\item[{(C3)}] $f$ satisfies Assumption \ref{assum_struc}, $g$ is the group LASSO regularizer, i.e., $g(x) := \sum_{J \in \mathcal{J}}\omega_J\|x_J\|_2$, where $\omega_J \geq 0$ and $\mathcal{J}$ is a partition of $\{1,\ldots,n\}$, and the optimal solution set $\cal X$ is compact.
\item[{(C4)}] $f$ is quadratic, $g$ is polyhedral convex.
	\end{itemize}

Notice that the Luo-Tseng error bound (\ref{Luo-Tseng-error-bound}) is only defined for the case where the function $g$ is convex and hence the convergence rate analysis based on the Luo-Tseng error bound can be only used to study the case where $g$ is convex. Recently the celebrated {\em{Kurdyka-{\L}ojasiewicz}} (KL) property (see, e.g., \cite{bolte2007lojasiewicz,bolte2010characterizations}) has attracted much attention in the optimization community.
In the fully nonconvex setting where both $f$ and $g$ are nonconvex, under the KL property, it has been shown that the PG method converges to a stationary point  which lies in the set of limiting stationary point defined as
\begin{equation*} \label{limitings}  {\cal{X}}^L := \left\{x\, \big |\, 0\in \nabla f( x)+\partial g( x)\right\}, \end{equation*}
where $\partial g( x)$ is the limiting subdifferential defined as in Definition \ref{defi-subdiff} (see, e.g. \cite{AttouchBolte2013Convergence,Bolte2014}).
  In particular, it has been shown that if $F$ is a coercive KL  function with an exponent $\frac{1}{2}$ as defined in the following definition, the sequence generated by the PG method converges linearly to a limiting stationary point in ${\cal X}^L$.
(see \cite[Theorem 3.4]{FrankelGarrogos2015}).
\begin{definition}
For a proper closed function $\phi:{\bbR^n} \rightarrow (-\infty, \infty]$ and $\bar{x} \in$ dom$\phi$, we say that $\phi$ satisfies  the KL property at $\bar{x}$ with an exponent of $\frac{1}{2}$ if there exist $\kappa, \epsilon> 0$ such that
			\begin{equation*}
			\kappa \left(\phi(x)-\phi(\bar{x})\right)^{-\frac{1}{2}} \,{\hbox{dist}}(0,\partial \phi(x)) \ge 1,\ \forall x \in \mathbb{B}(\bar{x},\epsilon) \cap \{x ~|~ \phi(\bar{x}) < \phi(x) < +\infty \},
			\end{equation*}
 where $\partial \phi$ is the limiting subdifferential of $\phi$ and $\mathbb{B}(\bar{x},\epsilon)$ is the open ball centered at origin with radius $\epsilon$. A proper closed function $\phi$ satisfying the KL property at all points in dom$\partial \phi$ is called a KL function.
\end{definition}

The KL property with an exponent of $\frac{1}{2}$ is very powerful in the linear convergence analysis since it leads to the linear convergence for  the fully nonconvex problem  and  the proper separation condition (\ref{isocost}) is not required. The drawback of this approach is that it is not easy to verify whether a function satisfies the KL property with an exponent of $\frac{1}{2}$ or not. In order to make use of  the known KL functions with an exponent of $\frac{1}{2}$ in producing more KL functions with an exponent of $\frac{1}{2}$,
recently, \cite{LiPong2016calculus} studies various calculus rules for the KL exponent. In particular,
\cite[Theorem 4.1]{LiPong2016calculus} shows that  the Luo-Tseng error bound and the proper separation condition (\ref{isocost}) implies that  $F$ is a KL function with an exponent of $\frac{1}{2}$ (similar results can be found in \cite{drusvyatskiy2016nonsmooth}). This implication consequently covers the results in
\cite{Wen2017linear}. Building upon the calculus rules and the connection
with the Luo-Tseng error bound, which is known
to hold for under of of conditions  (C1)-(C4),
\cite{LiPong2016calculus} shows that some optimization models with underlying structures have objectives whose KL exponent is $\frac{1}{2}$.

In this paper, we focus on the fully nonconvex case where both $f$ and $g$ are not necessarily convex. Suppose that $x_0$ is a local optimal solution of problem (\ref{Basic_Problem}). Then it follows easily from the definition of the proximal subdifferential (see  Definition \ref{defi-subdiff}) that $0 \in \partial^\pi F(x_0)$, where $\partial^\pi F (x)$ represents the proximal subdifferential of $F$ at $x$. By virtue of Assumption \ref{assum_sc_h}, the function $f$ is continuously differentiable with Lipschitz gradients in its domain and consequently by  the calculus rules in Proposition \ref{calculus2} we have
$$0 \in \nabla f(x_0) + \partial^\pi g (x_0) .$$
Based on this observation, we may define the set of proximal stationary points as follows.
\begin{equation*}\label{ProOptimal_solution_set}
{{\cal{X}}}^\pi: = \left\{ x \ \big|\ 0 \in \nabla f(x) + \partial^\pi g (x) \right\}.
\end{equation*}
Since $\partial^\pi g(x) \subseteq \partial g(x)$ and the inclusion may be strict,  in general  ${\cal{X}}^\pi\subseteq  {\cal{X}}^L $ and   they coincide when $\partial^\pi g(x)=\partial g(x)$. Hence the optimality condition defining the proximal stationary set ${\cal{X}}^\pi$ provides a shaper necessary optimality condition for problem (\ref{Basic_Problem})   than the one for the limiting stationary set ${\cal{X}}^L$ while in the case where $g$ is semi-convex (see Definition \ref{Prop2.5}),
${\cal X}^\pi={\cal X}^L$ and when $g$ is convex, ${\cal X}={\cal X}^\pi={\cal X}^L$.
 Since the stationary condition in terms of the proximal subdifferential provides a shaper necessary optimality condition, it is natural to ask whether one can prove that the PG method converges to a proximal stationary point. To our knowledge, all  results in the literature for the case where  $g$ is nonconvex prove the convergence to a limiting stationary point  (see, e.g., \cite{AttouchBolte2013Convergence,Bolte2014,FrankelGarrogos2015}). In this paper, we  show that the PG method actually converges to the set of the proximal stationary points ${\cal{X}}^\pi$ instead of the set of  the limiting stationary points ${\cal{X}}^L$.

Despite widespread use of the Luo-Tseng error bound in the convergence analysis and the calculus rules for the KL exponent, the concrete cases where the Luo-Tseng error bound holds are still very limited.
Motivated by this observation, in the recent paper \cite{ZhouSo2017},
 the authors relate the Luo-Tseng error bound condition to some unified sufficient conditions under the  convex setting.
They verify the existence of the Luo-Tseng error bound in concrete applications under the dual strict complementarity assumption.
However, the unified approach in \cite{ZhouSo2017} leads no improvement to the cases (C1)-(C3).  Moreover an extra compactness assumption of the optimal solution set  is even required for the case (C2).
The recent paper \cite{Drusvyatskiy2016error} further illuminates and extends some of the results in \cite{ZhouSo2017} by dispensing with strong convexity of component functions.

The limited application of the Luo-Tseng error bound is perhaps {due to} the fact that, from the theory of error bounds, except for the case where the Hoffman's error bound holds or the Robinson's polyhedral multifunction theory \cite{Robinson1980Strongly} holds, most of the sufficient conditions for error bounds are point-based (depending on the point of interest). The advantage of using point-based error bound is the existence of well-studied verifiable sufficient conditions.

The main goal of this paper is to find an appropriate point-based error bound type condition to meet  increasing needs in convergence analysis, calculus for the KL exponent and other aspects whenever appropriate. The new condition is regarded as both
	\begin{itemize}
\item a weaker replacement for the Luo-Tseng error bound condition when $g$ is convex,
\item an extension of the Luo-Tseng error bound condition when $g$ is nonconvex.
	\end{itemize}
In particular, the new condition meets the following requirements simultaneously.
	\begin{itemize}
\item[{(R1)}] It estimates
the distance to a chosen set of stationary points  in terms of certain easily computable residue.
\item[{(R2)}] In the fully nonconvex setting, together with weaker  point-based version of the proper separation condition, it ensures the linear convergence of the PG method toward the set of chosen stationary points.
\item[{(R3)}] It serves as a sufficient condition for the KL property with an exponent of $\frac{1}{2}$ under some mild assumptions.
\item[{(R4)}] It is generally weaker than the Luo-Tseng error bound condition when $g$ is convex and more importantly, it is easier to  verify via variational analysis when $g$ is nonconvex.
\item[{(R5)}] In the full convex setting, it results in some improvements to the linear convergence of the PG method for some of  cases of  (C1)-(C3).
	\end{itemize}

Although we conduct most of our analysis on the PG method, actually we are interested in  extensions to other first-order algorithms. A natural question to be answered is:

\noindent
\textbf{Q: For analyzing the convergence behavior of a given first-order algorithm, how to determine an appropriate type of error bound condition?}

In this paper, using the PG method as an example, we introduce a new perturbation analysis technique in order to provide an answer to  the above {question}. By the PG iteration scheme (\ref{PG_method_basic}), using  the sum rules of proximal subdifferentials in Proposition \ref{calculus2}  we obtain
\begin{equation}\label{eqn4.2NEW}
 \frac{p_{k+1}}{\gamma} \in \nabla f \left( x^{k+1} + p_{k+1} \right) + \partial^\pi g \left(x^{k+1} \right ),
 \end{equation}
where  $p_{k+1}:=x^k - x^{k+1}$.
Inspired by (\ref{eqn4.2NEW}), we define  the set-valued map induced by the PG method
\begin{equation*}\label{SPG}
{\cal{S}}_{PG} \left(p\right) := \left\{x \ \big|\  \frac{p}{\gamma} \in \nabla f \left( x + p \right) + \partial^\pi g \left(x\right)\right\}.
\end{equation*}
If the set-valued map ${\cal S}_{PG}$ has a stability property called  {\em{calmness}} around $(0, \bar x)$, where $\bar x\in {\cal X}^\pi$ is a accumulation point of the sequence $\{x^k\}$, then  exist $\kappa>0$ and a neighborhood  $\mathbb{U}(\bar x)$ of $\bar x$ such that
\begin{equation*}
d(x, {\cal X}^{\pi}) \leq \kappa \|p\|, \quad \forall  x\in \mathbb{U}(\bar x)\cap  {\cal S}_{PG}(p).
\end{equation*} The calmness of ${\cal{S}}_{PG}$ turns out to ensure that
\begin{equation*}
{\hbox{dist}}\left( x^{k+1}, {{\cal{X}}}^\pi \right)\le \kappa \left\|x^k - x^{k+1} \right\|,
\end{equation*}
which is essentially required for the linear convergence of the PG method toward ${\cal{X}}^\pi$.

 The calmness for a set-valued map is a fundamental concept  in variational analysis; see, e.g., \cite{henrion2002calmness,henrion2005calmness}. Although the terminology of ``calmness'' was coined by Rockafellar and Wets in \cite{Rockafellar2009variational}, it was first introduced in Ye and Ye \cite[Definition 2.8]{Ye1997Necessary} as the pseudo upper-Lipschitz continuity taking into account that the calmness is weaker than both the pseudo-Lipschitz continuity of Aubin \cite{Aubin1984Lipschitz} and the upper-Lipschitz continuity of Robinson \cite{Robinson1975Stability}. Therefore the calmness condition can be verified by either the polyhedral multifunction theory of Robinson \cite{Robinson1980Strongly} or by the Mordukhovich criteria based on the limiting normal cone \cite{Mordukhovich2006variational}.
 Recently based on the directional limiting normal cone, weaker verifiable sufficient conditions for calmness have been established (see, e.g. \cite[Theorem 1]{Gfrerer2017NewCQ}).

The perturbation analysis technique motivates us to use  the calmness of the set-valued map ${\cal{S}}_{PG}$ in analyzing the linear convergence of the PG method instead of the stringent Luo-Tseng error bound.
However, in variational analysis, most of the sufficient conditions for calmness are given for   a canonically perturbed system, while ${\cal{S}}_{PG}$ is not.
In order to connect the calmness of ${\cal{S}}_{PG}$ to the well-established theories in variational analysis, we further prove that
provided $\gamma <\frac{1}{L}$, the calmness of ${\cal{S}}_{PG}$ is equivalent to the calmness of the canonically perturbed stationary point set-valued map \begin{equation}\label{Basic_perturbation_multifunction}
{\cal{S}}_{cano} (p):= \left\{x \ \big|\ p \in \nabla f(x) + \partial^\pi g(x) \right\},
\end{equation}
or equivalently the metric subregularity of its inverse map ${\cal S}_{cano}^{-1}(x)=\nabla f(x)+\partial^\pi g(x)$. In this paper we will demonstrate that the sufficient conditions for the metric subregularity of ${\cal S}_{cano}^{-1}$ via variational analysis provide useful tools for convergence behavior analysis. In particular, in the fully convex setting, the calmness of ${\cal S}_{cano}$ is satisfied automatically for the cases (C1)-(C3) without any compactness assumption on the solution set ${\mathcal {X}}$. This observation further justifies the advantage of using calmness of ${\cal S}_{cano}$ as a replacement for the Luo-Tseng error bound in linear convergence analysis.

It is well-known that the calmness of ${\cal S}_{cano}$  is equivalent to the KL property with an exponent of $\frac{1}{2}$ in the case where $g$ is convex.
In this paper, we show that when $g$ is semi-convex, together with a point-based version of the proper separation condition (see (\ref{PBisocost}) in Assumption \ref{PBisocostassp}), the calmness of ${\cal S}_{cano}$ implies the KL property with an exponent of $\frac{1}{2}$.

The perturbation analysis idea for the PG method sheds some light on answering question \textbf{Q}. Indeed, the error bound condition we are looking for is the calmness condition for a perturbed set-valued map which is determined by  the iterative scheme of the given first-order algorithm with the perturbation parameter being the difference between the two consecutive generated points. To illustrate this point,  we investigate the following two examples of the first-order methods  in section 6.
	\begin{itemize}
\item[{(Ex1)}] We focus on the alternating direction method of multipliers (ADMM) for solving convex minimization model with linear constraints. The iteration scheme of ADMM introduces  the canonical type perturbation to the optimality condition instead of the proximal type. Taking advantage of a specific feature of ADMM's iterative scheme by which part of the perturbation is automatically zero, the perturbation analysis technique motivates a partial calmness condition. The partial calmness which derives the linear convergence rate of ADMM, is generally weaker than known error bound conditions in the literature.
\item[{(Ex2)}] We conduct some discussion on the perturbation induced by the iteration schemes of the primal-dual hybrid gradient (PDHG) method.
	\end{itemize}

Based on all discussions we summarize the main contributions of this paper as follows.
\begin{itemize}
\item We have shown that the PG method converges to a proximal stationary point. This observation has never been made in the literature before.
\item We justify that the calmness of ${\cal S}_{cano}$ defined as in (\ref{Basic_perturbation_multifunction}) can be regarded as an appropriate point-based improvement for the Luo-Tseng error bound condition taking into consideration that it meets requirements (R1)-(R5) simultaneously.

\item We propose a perturbation analysis technique for finding an appropriate error bound condition for the linear convergence analysis of the PG method. This technique is also applicable for carrying out the linear convergence analysis for various other first-order algorithms such as ADMM, PDHG, PPA and etc.

\end{itemize}

The remaining part of this paper is organized as follows. Section 2 contains notations and preliminaries. Section 3 briefs the linear convergence analysis for the PG method under the PG-iteration-based error bound. Section 4 starts with the introduction of a perturbation analysis technique which determines an appropriate type of calmness conditions to ensure the PG-iteration-based error bound.
Calmness conditions of the various perturbed stationary points set-valued maps and their relationship with the Luo-Tseng error bound condition and KL property will also be presented in Section 4. In Section 5, verification of the desired calmness condition for both  structured convex problems and general nonconvex problems are presented. Section 6 is dedicated to the application of the perturbation analysis technique to convergence behavior analysis of ADMM and PDHG.

\section{Preliminaries and preliminary results}\setcounter{equation}{0}

We first give notation that will be used throughout the paper. The open unit ball and closed unit ball around zero are given by $\mathbb{B}$ and $\overline{\mathbb{B}}$, respectively. $\mathbb{B}(\bar{x},r) := \{x \in {\bbR^d} ~|~ \|x-\bar{x}\| < r \}$ denotes the open ball around $\bar{x} \in \mathbb{R}^d$ with radius $r > 0$. For two vectors $a,b\in \mathbb{R}^d$, we denote by $\langle a,b \rangle $ the inner product of $a$ and $b$. For any $x\in \mathbb{R}^d$, we denote by $\|x\|$ its $l_2$-norm and $\|x\|_1$ its $l_1$-norm.  By $o(\cdot)$ we mean that $o(\alpha)/\alpha \rightarrow 0$ as $\alpha \rightarrow 0$. $x \stackrel{D}{\to} \bar x$ means that $x\rightarrow \bar x$ with all  $x\in D$. For a
{ differentiable} mapping $P: \mathbb{R}^d \rightarrow \mathbb{R}^s$ and a vector $x\in \mathbb{R}^d$, we denote by $\nabla P(x)$ the Jacobian matrix of $P$ at $x$ if $s>1$ and the gradient vector if $s=1$.  For a function $\varphi: \mathbb{R}^d \rightarrow \mathbb{R}$, we denote by $\varphi_{x_i} (x)$ and $\nabla^2 \varphi (x) $ the partial derivative of $\varphi$ with respect to $x_i$ and the  Hessian matrix of $\varphi$ at $x$, respectively. For a set-valued map $\Phi: \bbR^d \rightrightarrows \bbR^s$,
 the graph of  $\Phi $ is defined as $gph\,\Phi := \{(x,y) \in \bbR^d \times \bbR^s ~|~ y \in \Phi(x) \}$ and its inverse map  is defined by $\Phi^{-1}(y):=\{x \in \mathbb{R}^d: y\in \Phi(x)\}$.

For a set-valued map $\Phi: \mathbb{R}^d \rightrightarrows \mathbb{R}^s $, we denote by
\begin{eqnarray*}
\limsup_{x\rightarrow x_0} \Phi(x)&:= &\left \{\xi \in \mathbb{R}^s \big | \begin{array}{l}
\exists \mbox{ sequences } x_k \rightarrow x_0, \xi_k \rightarrow \xi,\\
\mbox{with } \xi_k \in \Phi(x_k) \ \ \forall k=1,2,\dots
\end{array}  \right \}
\end{eqnarray*}
the  Kuratowski-Painlev\'{e} upper (outer) limit. We say that a set-valued map $\Phi: \mathbb{R}^d \rightrightarrows \mathbb{R}^s$  is outer semicontinuous (\textit{osc}) at $x_0$ if
$\displaystyle
\limsup_{x\rightarrow x_0} \Phi(x) \subseteq \Phi(x_0).
$

\begin{definition}\cite[Definitions 8.45 and 8.3 and comments on page 345]{Rockafellar2009variational} \label{defi-subdiff}
Let $\phi : \mathbb{R}^d \rightarrow [-\infty, \infty ]$ and $x_0\in dom \phi$.
The proximal subdifferential of $\phi$ at $x_0$ is the set
$$
\partial^\pi \phi(x_0) := \left \{  \xi \in \mathbb{R}^d  \big | \begin{array}{l}
\exists \sigma>0,\eta>0 \mbox{ s.t. } \\
 \phi(x) \geq \phi(x_0)+\langle \xi,x-x_0\rangle -\sigma \|x-x_0\|^2 \quad \forall x\in \mathbb{B}(x_0,\eta)\end{array}
\right \}.
$$
The limiting (Mordukhovich or basic)
subdifferential of $\phi$ at $x_0$
is the closed set
\begin{eqnarray*}\partial \phi(x_0)
&=& \left \{\xi \in \mathbb{R}^d\big | \exists x_k \rightarrow x_0, \mbox{ and } \xi_k
\rightarrow \xi \mbox{ with } \xi_k \in \partial^\pi \phi(x_k), \phi(x_k) \rightarrow \phi (x_0)\right \}.
\end{eqnarray*}
\end{definition}
For any $x_0 \in dom \phi$, the set-valued map $\partial \phi$ is \textit{osc} at $x_0$ with respect to $x_k \rightarrow x_0$ satisfying $\phi(x_k) \rightarrow \phi(x_0)$ (see, e.g., \cite[Proposition 8.7]{Rockafellar2009variational}).
 In the case where $\phi$ is a convex function, all subdifferentials coincide with the subdifferential in the sense of convex analysis, i.e.,
$$ \partial^\pi \phi(x_0)= \partial \phi(x_0) =\left\{ \xi\in \mathbb{R}^d\, \big |\, \phi(x)-\phi(x_0)\geq \langle \xi, x-x_0\rangle, \quad \forall x  \right\}.$$

\begin{proposition}[Sum rule for limiting subdifferential] \label{calculus} \cite[Exercise 8.8]{Rockafellar2009variational}
Let $\varphi:\mathbb{R}^d \rightarrow \mathbb{R}$ be  differentiable at $x_0$ and $\phi:\mathbb{R}^d \rightarrow [-\infty,\infty]$ be finite at $x_0$. Then
$$
\partial (\varphi + \phi) (x_0) = \nabla\varphi (x_0) + \partial\phi (x_0).
$$
\end{proposition}
Inspired by \cite[Exercise 2.10 and Proposition 2.11]{Clark}, we present the following calculus rule.
\begin{proposition}[Sum  rule for proximal subdifferential] \label{calculus2}
	Let $\varphi:\mathbb{R}^d \rightarrow \mathbb{R}$ be differentiable at $x_0$ with  $\nabla \varphi$   Lipschitz continuous around $x_0$ and $\phi:\mathbb{R}^d \rightarrow [-\infty,\infty]$ be finite at $x_0$. Then
	$$
	\partial^\pi (\varphi + \phi)(x_0) = \nabla\varphi (x_0) + \partial^\pi \phi (x_0).
	$$
\end{proposition}
\begin{proof} Suppose that $\nabla \varphi$   Lipschitz continuous around $x_0$ with Lipschitz constant $\sigma'>0$. Then  there exists positive number $\eta'$ such that
\begin{equation}\label{C1+}
|\varphi(x) - \varphi(x_0)-\langle \nabla \varphi(x_0),x-x_0\rangle |\leq  \sigma' \|x-x_0\|^2, \quad \forall x\in \mathbb{B}(x_0,\eta').
\end{equation}
It follows straightforwardly from Definition \ref{defi-subdiff} and (\ref{C1+}) that
$$
	 \nabla\varphi (x_0) + \partial^\pi \phi (x_0) \subseteq \partial^\pi (\varphi + \phi)(x_0).
$$

Conversely, take any $\xi \in \partial^\pi (\varphi + \phi)(x_0)$. Then by definition,  there exist $\sigma>0,\eta>0$ such that
\begin{equation}\label{C1++}
\varphi(x) + \phi(x) \geq \varphi(x_0) + \phi(x_0) +\langle \xi,x-x_0\rangle -\sigma \|x-x_0\|^2, \quad \forall x\in \mathbb{B}(x_0,\eta).
\end{equation}
It follows from (\ref{C1+})-(\ref{C1++}) that we have
$$
\phi(x) \geq  \phi(x_0) +\langle \xi-\nabla \varphi(x_0),x-x_0\rangle -(\sigma'+\sigma) \|x-x_0\|^2, \quad \forall x\in \mathbb{B}(x_0,\min\{\eta',\eta\}),
$$
which implies $\xi-\nabla \varphi(x_0) \in \partial^\pi \phi (x_0)$ and thus we get the conclusion.
\end{proof}
We introduce a local version of a semi-convex function, see, e.g., \cite[Definition 10]{bolte2010characterizations}.
\begin{definition}Let $\phi : \mathbb{R}^d \rightarrow [-\infty, \infty ]$ and $x_0\in dom \phi$. We say $\phi$ is semi-convex around $x_0$ {with modulus $\rho>0$} if there exists $\eta>0$ such that the function
$\phi(x)+\frac{\rho}{2}\|x\|^2$ is convex on $\mathbb{B}(x_0,\eta)$. We say $\phi$ is semi-convex if it is semi-convex at every point in $dom \phi$ with unified modulus.
\end{definition}

The following result follows from  the calculus rules in Propositions \ref{calculus} and \ref{calculus2} immediately.
\begin{proposition}\label{Prop2.5}
Let $\phi : \mathbb{R}^d \rightarrow [-\infty, \infty ]$ and $x_0\in dom \phi$. If $\phi$ is semi-convex around $x_0$ on $\mathbb{B}(x_0,\eta)$ with $\eta>0$, then $\partial^\pi \phi(x)=\partial \phi(x)$ for all $x\in \mathbb{B}(x_0,\eta)$.
\end{proposition}

\begin{definition}\label{limitingnormalcone} \cite{Mordukhovich2006variational,Rockafellar2009variational} Let $D \subseteq {\bbR^d}$ and   $\bar{x}\in D$. The (Bouligand-Severi) tangent/contingent cone to $D$ at $\bar x$ is a  cone defined as
$$
{\cal{T}}_{D} \left(\bar{x}\right):= \left\{u \in {\bbR^d}\ \big|\ \exists\ u^k \rightarrow u, t^k \downarrow 0\ {\mbox{with}}\ \bar{x} + t^k u^k \in D\ {\hbox{for each}}\ k \right\}.
$$
The regular normal cone to $D$ at $\bar{x}$ is defined as
	$$
	{\widehat{\cal{N}}}_{D} \left(\bar{x}\right):= \left\{v \in  {\bbR^d}\ \big|\ \langle v, x - \bar{x} \rangle \leq o \left( \left\| x - \bar{x} \right\| \right) \ {\hbox{for each}}\ x\in D \right\}.
	$$The limiting normal cone to $D$ at $\bar{x}$ is defined as
	$$
	{\cal{N}}_{D} \left(\bar{x}\right):= \limsup_{x\stackrel{D}{\to}\bar{x} }{\widehat{\cal{N}}}_{D} \left( \bar{x}\right) = \left\{ v \in {\bbR^d}\, \big| \ \exists \ x^k \stackrel{D}{\to} \bar{x}, v^{k}\rightarrow
	 v\
	\mbox{with}\ v^k \in {\widehat{\cal{N}}}_{D} \left(x^k\right) \ {\hbox{for each}}\  k \right\}.$$
	\end{definition}
Recently, {a directional version of the  limiting normal cone was introduced in \cite{DirNorCone} and extended to general Banach spaces by Gfrerer \cite{Gfrerer2013directional}.}
\begin{definition}
	Let $D \subseteq {\bbR^d}$ and $\bar x\in D$.  Given $d \in {\bbR^d}$, the directional limiting normal cone to $D$ at $\bar{x}$ in the direction $d$ is defined as
	$$
	{\cal{N}}_D \left(\bar{x}; d\right) = \left\{ v \in {\bbR^d}\, \big| \ \exists \,
	t^k \downarrow 0,  v^k \rightarrow v, d^k \rightarrow d \mbox{ with } v^k \in {\widehat{\cal{N}}}_D \left(\bar{x} + t^k d^k\right) \ {\hbox{for each}}\ k \right\}.
	$$
\end{definition}	
By definition, it is easy to see that $	{\cal{N}}_D \left(\bar{x}; d\right)\subseteq	{\cal{N}}_D \left(\bar{x}\right)$ and $	{\cal{N}}_D \left(\bar{x}; 0\right)=	{\cal{N}}_D \left(\bar{x}\right)$.

All these definitions are fundamental in variational analysis, and the following lemma will be useful in this paper.
\begin{lemma}\label{lemma3.5}\cite[Proposition 3.3]{YeZhou2017}\label{TanDerivable} Let $D \subseteq {\bbR^d},D = D_1 \times \ldots \times D_m $ be the Cartesian product of the closed sets ${D}_i$ and $\bar{x} = (\bar{x}_1, \ldots, \bar{x}_m) \in D$. Then
$${\mathcal{T}}_{D} \left(\bar{x}\right) \subseteq {\mathcal{T}}_{D_1} \left(\bar{x}_1\right) \times \ldots \times {\mathcal{T}}_{D_m} \left(\bar{x}_m\right),$$
and for every $d = (d_1, \ldots, d_m)\in {\mathcal{T}}_{D} \left(\bar{x}\right)$ one has
$${\mathcal{N}}_{D} \left(\bar{x}; d\right) \subseteq {\mathcal{N}}_{D_1} \left(\bar{x}_1; d_1\right) \times \ldots \times {\mathcal{N}}_{D_m} \left(\bar{x}_{m}; d_{m}\right).$$
Furthermore, equalities hold in both inclusions if all except at most one of $D_i$ for $i= 1, \ldots, m$, are directionally regular
at $\bar{x}_i$ (see \cite[Definition 3.3]{YeZhou2017} for the definition of directional regularity). In particular, a set that is either convex or the union of finitely many convex polyhedra sets is directionally regular. Moreover the second-order cone complementarity set is also shown to be directionally regular in \cite[Theorem 6.1]{YeZhou2017}.
\end{lemma}
Next we review some concepts  of  stability of a set-valued map.
\begin{definition}
\cite{Robinson1975Stability}  A set-valued map ${\cal S}: \bbR^s \rightrightarrows \bbR^d$ is said to be upper-Lipschitz  around $(\bar p, \bar x) \in gph {\cal S}$  if there exist a neighborhood $\mathbb{U}(\bar p)$ of $\bar p$ and $\kappa\geq 0$ such that
{\rm{\begin{equation}
{\cal S} (p)   \subseteq {\cal S} (\bar p) + \kappa \left\| p-\bar p\right\| \overline{\mathbb{B}},\ \ \forall p \in \mathbb{U}(\bar p).
\end{equation}}}
\end{definition}
\begin{definition}\label{pseudo}
\cite{Aubin1984Lipschitz} A set-valued map ${\cal S}: \bbR^s \rightrightarrows \bbR^d$ is said to be pseudo-Lipschitz (or locally Lipschitz like or has the Aubin property) around $(\bar p, \bar x) \in gph  {\cal S} $   if there exist a neighborhood $\mathbb{U}(\bar p)$ of $\bar p$, a neighborhood $\mathbb{U}(\bar x)$ of $\bar x$ and $\kappa\geq 0$ such that
{\rm{\begin{equation}
{\cal S}(p)  \cap \mathbb{U}(\bar x) \subseteq {\cal S}(p') + \kappa \left\| p-p'\right\| \overline{\mathbb{B}}, \ \ \forall p, p'\in \mathbb{U}(\bar p).
\end{equation}}}Equivalently, ${\cal S}$ is pseudo-Lipschitz  around $(\bar p, \bar x) \in gph {\cal S}$ if there exist a neighborhood $\mathbb{U}(\bar p)$ of $\bar p$, a neighborhood $\mathbb{U}(\bar x)$ of $\bar x$ and $\kappa\geq 0$ such that
$$d(x, {\cal S}(p')) \leq \kappa d(p', {\cal S}^{-1}(x)), \quad \forall p' \in \mathbb{U}(\bar p), x\in \mathbb{U}(\bar x),$$
i.e., the inverse map ${\cal S}^{-1}$ is metrically regular around $(\bar x, \bar p)$.
\end{definition}
Both upper-Lipschitz continuity and the pseudo-Lipschitz continuity are stronger than the following concept which plays a key role in analyzing the linear convergence of some algorithms.
\begin{definition}\label{calm}
\cite{Ye1997Necessary,Rockafellar2009variational} A set-valued map ${\cal S}: \bbR^s \rightrightarrows \bbR^d$ is said to be calm (or pseudo upper-Lipschitz continuous)  around $(\bar p, \bar x) \in gph \mathcal{S}$ if there exist a neighborhood $\mathbb{U}(\bar p)$ of $\bar p$, a neighborhood $\mathbb{U}(\bar x)$ of $\bar x$ and $\kappa\geq 0$ such that
{\rm{\begin{equation}\label{calmness-defi}
{\cal S}(p) \cap \mathbb{U}(\bar x) \subseteq {\cal S}(\bar p) + \kappa \left\| p-\bar p\right\| \overline{\mathbb{B}}, \ \ \forall p \in \mathbb{U}(\bar p).
\end{equation}}}Equivalently, ${\cal S}$ is calm  around $(\bar p, \bar x) \in gph  {\cal S} $ if there exist a neighborhood $\mathbb{U}(\bar p)$ of $\bar p$, a neighborhood $\mathbb{U}(\bar x)$ of $\bar x$ and $\kappa\geq 0$ such that
\begin{equation}\label{Metric_subregularity_0}
{\rm{dist}}(x, {\cal S}(\bar p)) \leq \kappa\, {\rm{dist}}(\bar p, {\cal S}^{-1}(x)\cap \mathbb{U}(\bar p)), \quad \forall x\in \mathbb{U}(\bar x),
\end{equation}
i.e., the inverse map ${\cal S}^{-1}$ is metrically subregular around $(\bar x, \bar p)$.
\end{definition}
\begin{definition}\label{icalm}
\cite{Dontchev2009implicit} A set-valued map ${\cal S}: \bbR^s \rightrightarrows \bbR^d$ is said to be isolated calm around $(\bar p, \bar x) \in gph {\cal S}$ if there exist a neighborhood $\mathbb{U}(\bar p)$ of $\bar p$, a neighborhood $\mathbb{U}(\bar x)$ of $\bar x$ and $\kappa\geq 0$ such that
{\rm{\begin{equation*}
{\cal S}(p) \cap \mathbb{U}(\bar x) \subseteq \bar x + \kappa \left\| p-\bar p\right\| \overline{\mathbb{B}}, \ \ \forall p \in \mathbb{U}(\bar p).
\end{equation*}}}Equivalently, ${\cal S}$ is isolated calm  around $(\bar p, \bar x) \in gph {\cal S}$   if there exist
a neighborhood $\mathbb{U}(\bar p)$ of $\bar p$, a neighborhood $\mathbb{U}(\bar x)$ of $\bar x$ and $\kappa\geq 0$ such that
\begin{equation*}
\|x-\bar x\|\leq \kappa\, {\rm{dist}}(\bar p, {\cal S}^{-1}(x)\cap \mathbb{U}(\bar p)), \quad \forall  x\in \mathbb{U}(\bar x),
\end{equation*}
i.e., the inverse map ${\cal S}^{-1}$ is strongly metrical subregular around $(\bar x, \bar p)$.
\end{definition}
Note that by \cite[Exercise 3H.4]{Dontchev2009implicit}, the neighborhood $\mathbb{U}(\bar p)$ in Definitions \ref{calm} and \ref{icalm} can be equivalently replaced by the whole space $\mathbb{R}^d$.

Let
${\cal S}(p):=\{x\in {\bbR^d}| p\in -P(x)+ D\}$ where $ P(x): \mathbb{R}^d \rightarrow \mathbb{R}^s$  is locally Lipschitz and $D \subseteq \mathbb{R}^s$ is closed. Then the set-valued map ${\cal S}$ is {\em{calm}} at $(0, \bar x)$ if and only if ${\cal S}^{-1}$ is metrically subregular at $(\bar x, 0)$.  For convenience we summarize some verifiable sufficient conditions for the {\em{calmness}} of ${\cal S}$; see more criteria for calmness in \cite[Theorem 2]{Gfrerer2017NewCQ} and  \cite{henrion2002calmness,henrion2005calmness,YeZhou2017}.
\begin{proposition}\label{sufficient-MSCQ-new}
Let $\Omega:=\{ x\in {\bbR^d}| P(x)\in D\}$ where  $D$ is closed near $\bar x\in \Omega$. If $P(x)$ is continuously differentiable, let $\mathcal {T}^{lin}_\Omega \left(\bar{x}\right):=\{w\in \mathbb{R}^d|\nabla P(\bar x) w\in {\cal T}_D(P(\bar x))\}$ be the linearized cone of $\Omega$ at $\bar x$. Then the set-valued map ${\cal S}(p):=\{x\in \mathbb{R}^d| p\in -P(x)+ D\}$
is calm at $(0, \bar x)$ if one of the following condition holds.
\begin{enumerate}
	\item[1.] Linear CQ holds (see, e.g.,  \cite[Theorem 4.3]{Ye2000constraint}): $P(x)$ is piecewise affine and $D$ is the union of finitely many convex polyhedra sets.
	
	\item[2.] No nonzero abnormal multiplier constraint
	qualification (NNAMCQ) holds at $\bar x$ (see, e.g., \cite[Theorem 4.4]{Ye2000constraint})
	$$
	0\in \partial \langle P, \lambda \rangle (\bar{x}), \,\, \lambda \in \mathcal {N}_{D} (P(\bar{x})) \quad \Longrightarrow \quad \lambda =0.
	$$
	
\item[3.]
	First-order sufficient condition for metric subregularity (FOSCMS) at $\bar x$ for the system $P(x) \in D$ at $\bar x$  with $P$ continuously differentiable at $\bar{x}$ \cite[Corollary 1]{Gfrerer2016Lipschitz}: for every $0 \neq w \in \mathcal {T}^{lin}_\Omega \left(\bar{x}\right)$,  one has
	$$
	\nabla P(\bar{x})^T \lambda = 0, \,\, \lambda \in \mathcal {N}_{D} (P(\bar{x}); \nabla P(\bar{x}) w) \quad \Longrightarrow \quad \lambda =0.
	$$
	
	\item[4.] Second-order sufficient condition for metric subregularity (SOSCMS)  at $\bar x$ for the system $P(x) \in D$ with P twice differentiable at $\bar{x}$ and $D$ is the union of finitely many convex polyhedra sets \cite[Theorem 2.6]{gfrerer2014optimality}: for every $0 \neq w \in \mathcal {T}^{lin}_\Omega \left(\bar{x}\right)$ one has
	$$
	\nabla P(\bar{x})^T \lambda = 0, \,\, \lambda \in \mathcal {N}_{D} (P(\bar{x}); \nabla P(\bar{x}) w), \,\, w^T \nabla^2 (\lambda^T P)(\bar{x}) w \geq 0 \quad \Longrightarrow \quad \lambda =0.
	$$
\end{enumerate}
\end{proposition}
\begin{remark}\label{remark}
Recall that a set-valued map is called a polyhedral multifunction if its graph is the union of finitely many polyhedral convex sets. For the case Proposition \ref{sufficient-MSCQ-new}(1), since the set-valued map ${\cal S}$ is a polyhedral multifunction, by \cite[Proposition 1]{Robinson1980Strongly},  the set-valued map is upper-Lipschitz and hence calm around every point of the graph  of ${\cal S}$. By the Mordukhovic criteria (see, e.g., \cite{Mordukhovich2006variational}),
NNAMCQ implies the pseudo-Lipschitz continuity of the set-valued map ${\cal S}$ (see e.g., \cite[Theorem 4.4]{Ye2000constraint}).
FOSCMS in  (3) holds automatically if
$$w \in \mathcal {T}^{lin}_\Omega \left(\bar{x}\right) \Longrightarrow w=0.$$
In this case, according to  the graphical derivative criterion for strong metric subregularity (see e.g., \cite{Dontchev2009implicit}),  the set-valued map ${\cal S}$ is in fact isolated calm. SOSCMS is obvious weaker than FOSCMS in general. Since the directional normal cone is in general a smaller set than the limiting normal cone, FOSCMS is in general weaker than NNAMCQ. But in the case when either $\nabla P(\bar x)$ does not have full column rank or $D$ is convex and there is $w\not =0$ such that $w\in \mathcal {T}^{lin}_\Omega \left(\bar{x}\right)$, then FOSCMS is equivalent to NNAMCQ \cite[Theorem 4.3]{YeZhou2017}. When $D=D_1\times \cdots \times D_{m}$ is the Cartesian product of closed sets $D_i$, by Lemma
\ref{TanDerivable}, we have for $\bar{x} = (\bar{x}_1, \ldots, \bar{x}_m) \in D$, $d = (d_1, \ldots, d_{m})\in {\mathcal{T}}_{D} \left(\bar{x}\right)$
\begin{eqnarray*}
&& {\mathcal{T}}_{D} \left(\bar{x}\right) \subseteq {\mathcal{T}}_{D_1} \left(\bar{x}_1\right) \times \ldots \times {\mathcal{T}}_{D_{m}} \left(\bar{x}_{m}\right),\\
&& {\mathcal{N}}_{D} \left(\bar{x}; d\right) \subseteq {\mathcal{N}}_{D_1} \left(\bar{x}_1; d_1\right) \times \ldots \times {\mathcal{N}}_{D_{m}} \left(\bar{x}_{m}; d_{m}\right).
\end{eqnarray*} Therefore we may replace $ {\cal T}_D(P(\bar x))$ and $\mathcal {N}_{D} (P(\bar{x}); P(\bar{x}) w)$
by ${\mathcal{T}}_{D_1} \left(P_1(\bar{x})\right) \times \ldots \times {\mathcal{T}}_{D_{m}}
\left(P_m(\bar{x})\right)$ and ${\mathcal{N}}_{D_1} \left(P_1(\bar{x}); [P(\bar{x}) w]_1\right) \times \ldots \times {\mathcal{N}}_{D_{m}} \left(P_m(\bar{x}); [P(\bar{x}) w]_{m}\right)$, where $[P(\bar{x}) w]_{i}$ denotes  the $i$th component of the vector $P(\bar{x}) w$, respectively,  to obtain a sufficient condition for calmness.
These types of sufficient conditions would be stronger in general but equivalent to the original one if all expect at most one of the sets $D_i$ is directionally regular.
\end{remark}

We close this section with the following equivalence. Proposition \ref{equi_ms} improves the result \cite[Proposition 3]{Gfrerer2017NewCQ} in that ${gph}Q$ is not assumed to be closed. When ${gph}Q$ is not closed, the projection onto ${gph}Q$ may not exist.  However if one replaces the projection with an approximate one, then the proof would go through and so we omit the proof.
\begin{proposition}\label{equi_ms}
Let
$ P(x): \mathbb{R}^d \rightarrow \mathbb{R}^s$ and $Q: \mathbb{R}^d \rightrightarrows \mathbb{R}^s$ be a set-valued map. Assume that $P$ is  Lipschitz around $\bar{x}$, then the set-valued map $M_1(x):=P(x)+Q(x)$ is metrically subregular at $(\bar{x},0)$ if and only if the set-valued map
\begin{equation*}
M_2(x):=\begin{pmatrix}
x \\ -P(x)
\end{pmatrix} - {gph}Q
\end{equation*}
is metrically subregular at $(\bar{x},(0,0))$.
\end{proposition}

\section{Linear convergence under the PG-iteration-based error bound}\label{Sec3}\setcounter{equation}{0}
The purpose of this section is to obtain the linear convergence result (\ref{linear-convergence-objective})-(\ref{linear-convergence-sequence}) under a weak and basic error bound assumption (\ref{error-bound-condition-1}) along with a  proper separation of the stationary value condition around the accumulation point  (\ref{PBisocost}). The main result is summarized in Theorem \ref{Thm3.2}.
As a prerequisite of the analysis to be delineated, Lemma \ref{sufficient-cost-to-go} can be derived similarly as  related
results in the literature, see, e.g., \cite[Section 5.1]{AttouchBolte2013Convergence}. We state the results now and leave the   proof of Lemma \ref{sufficient-cost-to-go} to the appendix.

\begin{lemma}\label{sufficient-cost-to-go}
Let $\{ x^k \}$ be a  sequence generated by the PG method. Suppose that  $\gamma <\frac{1}{L}$. Then the following statements are true.
\begin{enumerate}
\item[(1)] {\em{Sufficient descent}}: there exists a constant $\kappa_1 > 0$ such that
\begin{equation}\label{Sufficient-descent-inequality-0}
F(x^{k+1}) - F(x^k)\le - \kappa_1 \left\| x^{k+1} - x^k \right\|^2.
\end{equation}
\item[(2)] {\em{Cost-to-go estimate}}: there exists a constant $\kappa_2>0$ such that
\begin{equation}\label{cost-to-go-estimate-0}
F(x^{k+1}) - F(x) \le \kappa_2 \left( \left\| x - x^{k+1} \right\|^2 + \left\| x^{k+1} - x^k \right\|^2 \right), \quad \forall x.
\end{equation}
\end{enumerate}
\end{lemma}
Note that the terminologies ``sufficient descent" and ``cost-to-go estimate" were first used in \cite{LuoTseng1993,Tseng2009coordinate}.
Based on the sufficient descent and the cost-to-go estimate properties in Lemma \ref{sufficient-cost-to-go}, which are two fundamental inequalities for the convergence proof of first order methods, we prove the linear convergence of the PG method.
Before presenting the result, we first discuss the assumptions needed. The first one is a local version of the  proper separation of isocost surfaces condition (\ref{isocost}).
\begin{assumption}\label{PBisocostassp}
We say that the proper separation of isocost surfaces of $F$ holds on $\bar{x} \in {{\cal X}}^\pi$ if
\begin{equation}\label{PBisocost}
 \exists \epsilon > 0 \mbox{ such that } x \in {{\cal X}}^\pi\cap {\mathbb{B}} ( \bar{x},\epsilon )\quad \Longrightarrow \quad F(x) = F(\bar{x}).
\end{equation}
\end{assumption}
It is obvious that
condition (\ref{PBisocost}) is equivalent to
\begin{equation*}\label{PBisocostNew}
 \exists \epsilon > 0 , \delta>0 \mbox{ such that } x,y \in {{\cal X}}^\pi\cap {\mathbb{B}} ( \bar{x},\epsilon ) \mbox{ and } F(x)\not =F(y) \quad \Longrightarrow \quad \|x-y\|>\delta.
\end{equation*}
Hence condition (\ref{PBisocost}) is weaker than (\ref{isocost}).

This assumption holds whenever the objective function takes on only a finite number of values on ${\cal{X}}^\pi$ locally around $\bar{x}$, or whenever the connected components of ${\cal{X}}^\pi$ around $\bar{x}$ are properly separated from each other \cite{LuoTseng1993}.
Thus this assumption holds automatically when $\bar{x}$ is an isolated stationary point.
It also holds if the objective function $F$ is convex, or $f$ is quadratic and $g$ is polyhedral \cite{Luo2006Error}.

\begin{definition}[PG-iteration-based error bound]
	Let the sequence $\left\{ x^k \right\}$ be generated by the PG method and $\bar{x}$ is an accumulation point of $\left\{ x^k \right\}$. We say that the PG-iteration-based error bound holds at $\bar x$ if there exist $\kappa,\epsilon > 0$ such that
\begin{equation}\label{error-bound-condition-1}
 {\rm{dist}}\left( x^{k+1}, {\cal{X}}^\pi \right)\le \kappa \left\|x^{k+1} - x^k \right\|, \quad \text{for all } k \text{ such that}~ x^{k+1} \in {\mathbb{B}} ( \bar{x},\epsilon ).
\end{equation}
\end{definition}

Theorem \ref{Thm3.2} shows the linear convergence of the PG method under the PG-iteration-based error bound and the proper separation of stationary value (\ref{PBisocost}).
Most of the proof techniques and methodology are mild modifications  of the analysis based on the KL inequality (\cite{AttouchBolte2013Convergence}). However, a critical phenomenon which has been completely neglected in the literature leans on the fact that the PG method actually converges toward the proximal stationary set ${{\cal X}}^\pi$.
In spite of this interesting observation, we still leave the proof of Theorem \ref{Thm3.2} in Appendix, as our main concern is when the PG-iteration-based error bound can be met, which will be addressed in the forthcoming section.

\begin{theorem} \label{Thm3.2} \rm
Assume that the step-size $\gamma$ in the PG method \eqref{PG_method_basic} satisfies $\gamma <\frac{1}{L}$. Let the sequence $\{x^k\}_{k=0}^\infty$ be generated by the PG method and $\bar{x}$ be an accumulation point of $\left\{ x^k \right\}_{k=0}^\infty$. Then $\bar x \in {{\cal X}}^\pi.$  Suppose that the PG-iteration-based error bound  (\ref{error-bound-condition-1}) along with
 the proper separation of stationary value (\ref{PBisocost}) hold at $\bar x$.
 Then the sequence $\left\{ x^k \right\}_{k=0}^\infty$ converges to $\bar x$ linearly
with respect to the sequence of objective function values, i.e.,  there exist $k_0 > 0 $ and $0<\sigma < 1$, such that for all $k \ge k_0$, we have
\begin{eqnarray}
F(x^{k+1}) - F(\bar{x}) &\le& \sigma \left(F(x^{k}) - F(\bar{x})\right)
.\label{linear-convergence-objective}
\end{eqnarray}
Moreover we have for all $k \ge k_0$,
\begin{eqnarray}
\left\| x^k - \bar{x} \right\| &\le& \rho_0\, \rho^k,\label{linear-convergence-sequence}
\end{eqnarray}
for some $\rho_0>0, 0<\rho<1$.
\end{theorem}

\section{Discussions on various error bound conditions}
\setcounter{equation}{0}

In Section \ref{Sec3}, we have shown that linear  convergence of the PG method relies heavily on the PG-iteration-based error bound condition (\ref{error-bound-condition-1}).
In this section, we shall find an appropriate condition sufficient for the PG-iteration-based error bound condition \eqref{error-bound-condition-1} which is independent of the iteration sequence. For this purpose, we propose a new perturbation analysis technique which determines an appropriate error bound type condition for convergence analysis. In fact all results in this section remind true if one replace the proximal stationary point set ${\cal X}^\pi$ by  the  limiting stationary point set ${\cal X}^L$.

\subsection{A perturbation analysis technique}

We recall  by  the PG iteration scheme that given $x^k$, $x^{k+1}$ is an optimal solution to the optimization problem
$$\min_{x\in \mathbb{R}^n} \   \langle \nabla f(x^k),x-x^k\rangle +\frac{1}{2\gamma} \|x-x^k\|^2 +g(x).$$ By the calculus rule in Proposition \ref{calculus2},  since $f$ is smooth with $\nabla f$ Lipschitz near the point $x^{k+1}$,
\begin{equation}\label{opt_con_k}
0 \in  \nabla f (x^k)  + \frac{1}{\gamma} \left(x^{k+1} - x^{k} \right) + \partial^\pi g \left(x^{k+1}\right).
\end{equation}
Denote by $p_{k+1} := x^k - x^{k+1} $. Then the above inclusion can be rewritten as
\begin{equation}\label{Optimality_condition_new}
\frac{ p_{k+1}}{\gamma} \in \nabla f \left(x^{k+1} + p_{k+1} \right) + \partial^\pi g (x^{k+1}).
\end{equation}
It follows that  condition (\ref{error-bound-condition-1}) can be rewritten as
\begin{equation}\label{error-bound-condition-new}
{\hbox{dist}}\left( x^{k+1}, {\cal{X}}^\pi \right)\le \kappa \left\|p_{k+1}  \right\|, \quad \text{for all}~ k ~\text{s.t.}~ x^{k+1} \in {\mathbb{B}} ( \bar{x},\epsilon ).
\end{equation}
where $p_{k+1}$ satisfies (\ref{Optimality_condition_new}).
Inspired by  (\ref{Optimality_condition_new}), we define the following set-valued map induced by the PG method
\begin{equation}\label{Set_value_function_new}
{\cal{S}}_{PG} \left(p\right) := \left\{x \ \big|\  \frac{p}{\gamma} \in \nabla f \left( x + p \right) + \partial^\pi g \left(x\right)\right\},
\end{equation}

By Definition \ref{calm} and the comment after that, the set-valued map ${\cal S}_{PG}$ is {\em{calm}} around $(0, \bar x)$ if and only if there exist $\kappa>0$ and a neighborhood  $\mathbb{U}(\bar x)$ of $\bar x$ such that
\begin{equation}\label{Metric_subregularity_0new}
d(x, {\cal X}^\pi) \leq \kappa \|p\|, \quad \forall  x\in \mathbb{U}(\bar x), p\in {\cal S}_{PG}^{-1}(x).
\end{equation}
By taking $x=x^{k+1}$ and $p_{k+1} = x^k - x^{k+1}$ for sufficently large $k$ in the above condition, one can see that  the {\em{calmness}} of ${\cal S}_{PG}$ at $(0,\bar x)$ is a sufficient condition for condition (\ref{error-bound-condition-new}) or equivalently the PG-iteration-based error bound condition (\ref{error-bound-condition-1})  to hold and it is independent of the iteration sequence.

\subsection{Interplay between error bound conditions}
The question is now how to find verifiable sufficient conditions for the {\em{calmness}} of ${\cal S}_{PG}$ and what are the relationships with other related set-valued maps and the Luo-Tseng error bound.

In order to paint a complete picture, we   define the following three set-valued maps.
Firstly, by  taking $x^k=x^{k+1}+p_{k+1}, $ in (\ref{Optimality_condition_new}),  the PG also induces the following set-valued map:
$$
{\cal{S}}_{PGb} \left(p\right) := \left\{ x \ \big|\  \frac{p}{\gamma}  \in \nabla f \left( x  \right) + \partial^\pi g \left( x -p \right)\right\}.
$$
Secondly, we define the  set-valued map ${\cal{S}}_{PPA}$ induced by the PPA \footnote{The iteration scheme of PPA for problem (\ref{Basic_Problem}) can be written as $ x^{k+1} = \hbox{Prox}_{f + g}^{\gamma} \left( x^k \right)$. Straightforward calculation further implies
	$\frac{p_{k+1}}{\gamma} \in \nabla f \left(x^k - p_{k+1}\right) + \partial^\pi g \left( x^k - p_{k+1}\right)$ with $p_{k+1} = x^k - x^{k+1}$.} as
$$
{\cal{S}}_{PPA} \left(p\right) := \left\{ x \ \big|\  \frac{p}{\gamma}  \in \nabla f \left( x - p \right) + \partial^\pi g \left( x - p \right)\right\}.
$$
Note that $${\cal{S}}_{cano} (0)={\cal{S}}_{GP} (0)={\cal{S}}_{GPb} (0)={\cal{S}}_{PPA} (0)={{\cal X}}^\pi$$
and hence all these  four set-valued maps are solutions of the proximal stationary point set ${{\cal X}}^\pi$ perturbed in certain way.

On the other hand, we may define the following pointwise extension of the Luo-Tseng error bound (\ref{Luo-Tseng-error-bound}) in the general nonconvex case.
We say that the proximal error bound holds at $\bar x$ if
  there exist constants $\kappa > 0$ and $\epsilon > 0$, such that
\begin{equation}\label{Luo-Tseng-neighborhood-error-bound}
\exists\ \kappa, \epsilon > 0,\quad {\rm{dist}} \left(x, {{\cal{X}}}^\pi\right) \le \kappa {\rm{dist}} (x,  {\rm{Prox}}^{\gamma}_{g} \left( x - \gamma \nabla f(x) \right) ),\ \forall x \in \mathbb{B}\left(\bar{x},\epsilon\right).
\end{equation}
In fact, the pointwise extension of the Luo-Tseng error bound (\ref{Luo-Tseng-error-bound}), i.e., proximal error bound (\ref{Luo-Tseng-neighborhood-error-bound}) is nothing but the metric subregularity of the proximal residue $r(x)={\rm{dist}} (x,  {\rm{Prox}}^{\gamma}_{g} \left( x - \gamma \nabla f(x) \right) )$.

The connections we intent to prove can be illustrated in the following Figure \ref{flowchart}, justifying the promised (R4) in Section 1.

\begin{figure}[H]
	\begin{tikzpicture}[->,>=stealth',auto,node distance=10.5em,semithick]
	\tikzstyle{every state}=[rectangle,text=black]
	
	\node[state] (X) at (-4.5,0) {{\em{Calmness}} of ${\cal{S}}_{PPA}$ at $(0,\bar{x})$};
	\node[state] (Y) at (-1.75,-2) {{\em{Calmness}} of ${\cal{S}}_{cano}$ at $(0,\bar{x})$};
	\node[state] (T) at (1,0) {{\em{Calmness}} of ${\cal{S}}_{PG}$ at $(0,\bar{x})$};
	\node[state] (Z) at (6.5,0) {{\em{Calmness}} of ${\cal{S}}_{PGb}$ at $(0,\bar{x})$};
	\node[state] (W) at (4,2) {Proximal  Error Bound (\ref{Luo-Tseng-neighborhood-error-bound}) at $\bar{x}$};
	\node[state] (U) at (-4,2) {Luo-Tseng Error Bound (\ref{Luo-Tseng-error-bound})};
	\node[state] (V) at (4.5,-2) {Verifiable Sufficient Condition};
	\node[state] (A) at (0,4) {KL property with an exponent of $\frac{1}{2}$ at $\bar{x}$};
	
	\draw (X.-70) -- (Y.95) node[midway, above] {};
	\draw (Y.130) -- (X.-110) node[midway, above] {};
	\draw (Y.85) -- (T.-110) node[midway, above] {};
	\draw[dashed] (T.-70) -- (Y.55) node[midway, right] {$\quad\gamma\in (0,\frac{1}{L})$};
	\draw (T.5) -- (Z.175) node[midway, above] {};
	\draw (Z.185) -- (T.-5) node[midway, above] {};
	\draw (Z.70) -- (W.-70) node[midway, right] {};
	\draw[dashed] (W.-110) -- (Z.110) node[midway, left] {$\quad$ $g$ is semiconvex with modulus $\rho$ and $\gamma\in (0,\frac{1}{\rho}]\ \quad$};
	\draw[dashed] (U.0) -- (W.180) node[midway, above] {$g$ is convex};
	\draw (V.180) -- (Y.0) node[midway, above] {};
	\draw[dashed] (W.90) -- (A.-90) node[midway, right] {$\quad$ $g$ is semiconvex + Assumption \ref{PBisocostassp}+ $\gamma < \min\left\{\frac{1}{\rho},\frac{1}{L}\right\}$};
	\end{tikzpicture}\caption{Relationships between the {\em{calmness}} of $\mathcal {S}_{PPA}$, $\mathcal {S}_{cano}$, $\mathcal {S}_{PGb}$ and $\mathcal {S}_{PG}$,  the Luo-Tseng error bound condition (\ref{Luo-Tseng-error-bound}), the proximal error bound condition (\ref{Luo-Tseng-neighborhood-error-bound}) and KL property with an exponent of $\frac{1}{2}$. The dotted arrow means that extra conditions are required.}\label{flowchart}
\end{figure}
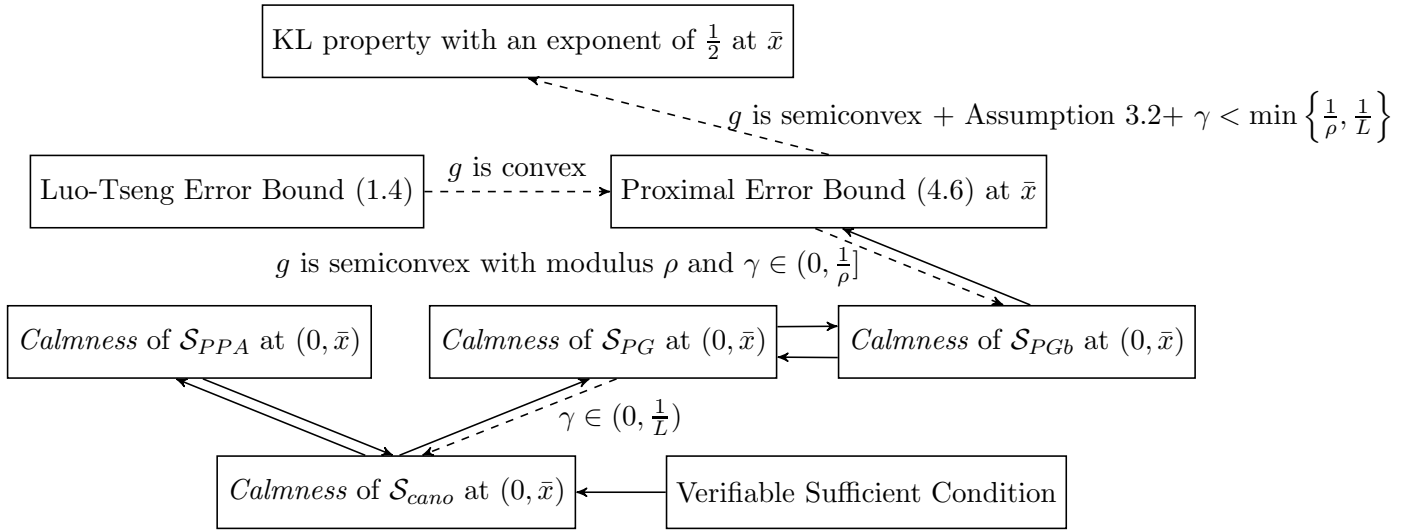

The following theorem clarifies all the details in Figure \ref{flowchart} except the implication that the verifiable sufficient condition  implies the calmness of $\mathcal {S}_{cano}$. This implication will be discussed Section \ref{Sec5}. Since the results
of the following theorem  are of independent interests, we will state the assumptions whenever needed instead of using the standing assumption \ref{assum_sc_h}.
\begin{theorem}\label{Full_theorem}
	Let $\bar{x}$ belong to ${{\cal{X}}}^\pi$.
	\begin{enumerate}
	\item[{\rm (i)}] The calmness of $\mathcal {S}_{PG}$ at $(0,\bar{x})$ is equivalent to the calmness of $\mathcal {S}_{PGb}$ at $(0,\bar{x})$.
		\item[{\rm (ii)}] Assume that $\nabla f$ is Lipschitz continuous on a neighborhood of $\bar x$ with constant $L>0$. Then
		the calmness of ${\cal{S}}_{cano}$ at $(0,\bar{x})$ implies the calmness of $\mathcal {S}_{PG}$ at $(0,\bar{x})$ and the reverse direction holds provided that  $\gamma<\frac{1}{L}$.
		\item[{\rm (iii)}] The calmness of $\mathcal {S}_{PPA}$ at $(0,\bar{x})$ is equivalent to the calmness of $\mathcal {S}_{cano}$ at $(0,\bar{x})$.
		\item[{\rm (iv)}]  The calmness of ${\cal{S}}_{PGb}$ at $(0,\bar{x})$ implies the proximal error bound condition (\ref{Luo-Tseng-neighborhood-error-bound}) at $\bar{x}$. The reversed direction also holds  when {$g$ is  semi-convex around $\bar x$}  with modulus $\rho$ and $\gamma \le \frac{1}{\rho}$.
		\item[{\rm (v)}] The  Luo-Tseng error bound condition (\ref{Luo-Tseng-error-bound}) implies the proximal error bound condition (\ref{Luo-Tseng-neighborhood-error-bound}) when $g$ is convex.
   \item[{\rm (vi)}] Assume that $\nabla f$ is Lipschitz continuous on a neighborhood of $\bar x$ with constant $L>0$.  When {$g$ is  semi-convex around $\bar x$}, if the proper separation of stationary value (\ref{PBisocost}) and the proximal error bound \eqref{Luo-Tseng-neighborhood-error-bound} at $\bar{x}$ holds, then $F$ satisfies the  KL property with an exponent of $\frac{1}{2}$ at $\bar{x}$
	\end{enumerate}
\end{theorem}
\begin{proof}
{\bf (i)}: Suppose  the {\em{calmness}} of ${\cal{S}}_{PG}$ at $(0,\bar{x})$ holds. It follows by definition that there exist a neighborhood ${{\mathbb{U}}}(\bar x)$ of $\bar{x}$ and $ \kappa >0$ such that
	\begin{equation}\label{calmSpg}
	{\hbox{dist}} \left( x, {\cal{S}}_{PG} \left(0\right)\right)  \le  \kappa  \left\|p\right\|,\quad \forall x \in {{\mathbb{U}}}(\bar x),  p\in {\cal{S}}_{PG}^{-1}(x).
	\end{equation}
Take arbitrary $x \in {\cal{S}}_{PGb} \left(p\right)$. Then by definition
	$$
	\frac{p}{\gamma} \in \nabla f(x) + \partial^\pi g (x - p),$$
which can be written as
	$$
	\frac{p}{\gamma} \in \nabla f(\tilde{x} + p) + \partial^\pi g (\tilde{x}),
	$$ with 	 $\tilde{x} := x - p$ and hence $\tilde{x} \in \mathcal {S}_{PG}(p)$.
	Let $ {{\mathbb{U}}}_0(\bar x)$ be a neighborhood of $\bar x$ and $\delta>0$ be such that
	$$  \tilde{x} = x-p \in {\mathbb{U}}(\bar x), \quad \forall x \in {{\mathbb{U}}}_0(\bar x), \left\| p \right\| \le \delta. $$ Then by  using (\ref{calmSpg}), for any $x \in {{\mathbb{U}}}_0(\bar x)$, $\left\| p \right\| \le \delta$ with $ p\in {\cal{S}}_{PGb}^{-1}(x)$, we have
	$$
	{\hbox{dist}} \left( x, {\cal{S}}_{PGb} \left(0\right)\right) \le {\hbox{dist}} \left( x - p, {\cal{S}}_{PGb}\left(0\right)\right) + \left\| p \right\| = {\hbox{dist}} \left( \tilde{x}, {\cal{S}}_{PG} \left(0\right)\right) + \left\|p\right\| \le \left( \kappa + 1 \right) \left\|p\right\|,
	$$which implies that ${\cal{S}}_{PGb} $ is {\em{calm}} at $(0,\bar{x})$.
The proof of the reverse direction is similar and hence  omitted.

	{\bf (ii)}:  We rewrite $\mathcal {S}_{PG}(p)$ as
	\begin{equation*}
	\mathcal {S}_{PG} \left(p\right) := \left\{ x \ \big|\ 0 \in {\cal{M}} \left(p, x\right) \right\},
	\end{equation*}where
	$$\mathcal M(p,x) := \mathcal G(p,x) + {gph} \left(\partial^\pi g\right),\ \mathcal G(p,x) := \begin{pmatrix}
	-x \\
	-\frac{p}{\gamma} + \nabla f \left( x + p \right)
	\end{pmatrix}.
	$$
	Following the technique presented in \cite{Gfrerer2016Lipschitzian}, we introduce two multifunctions $H_{\mathcal M} : \bbR^n \rightrightarrows \bbR^n \times \bbR^n \times \bbR^n$ and $\mathcal M_p : \bbR^n \rightrightarrows \bbR^n \times \bbR^n$ defined by
	$$H_{{\cal{M}}} \left(p\right) := \left\{ (x,y) \ \big|\ y \in {\cal{M}}\left(p,x\right) \right\}\quad \hbox{and}\quad {\cal{M}}_p \left(x\right) := \left\{y \ \big|\ y \in {\cal{M}} \left(p,x\right) \right\}.$$
	By  \cite[Theorem 3.3]{Gfrerer2016Lipschitzian}, if $\mathcal M_0(x) := \mathcal M(0,x)$ is metrically subregular at $(\bar{x}, 0)$ and  ${\cal{M}}$ has the {\em{restricted calmness property}} with respect to $p$ at $(0,\bar{x},0)$, i.e.,  if there are real numbers $\kappa > 0$ and $\epsilon > 0$ such that
	$$
	 {\rm{dist}} \left( \left( x,\bar{y}\right) , H_{\cal{M}} \left(0\right)\right) \le \kappa \, \left\| p \right\|, \quad \forall \left\| p \right\| \le \epsilon, \left\| x - \bar{x} \right\|\le \epsilon,\left( x,\bar{y}\right) \in H_{\cal{M}} \left(p\right),
	$$ then $\mathcal {S}_{PG}$ is {\em{calm}} at $(0, \bar{x})$.
	Based on this result we can show that the {\em{calmness}} of ${\cal S}_{cano}$ implies the {\em{calmness}} of $\mathcal {S}_{PG}$.
	\begin{itemize}
		\item We can show that ${\cal{M}}$ has the {\em{restricted calmness property}} with respect to $p$ at $(0,\bar{x},0)$.
		Indeed, since $\nabla f(x)$ is Lipschitz around $\bar x$ with  constant $L > 0$, there is a neighborhoods $\mathbb{U}(0)$ of $0$ as well as $\mathbb{U}(\bar x)$ of $\bar{x}$ such that $\mathcal G$ is also Lipschitz continuous with modulus $L$ on $\mathbb{U}(0) \times \mathbb{U}(\bar x) \subseteq  \bbR^n \times \bbR^n $. Given $(p,x,0)$ where $p \in \mathbb{U}(0), x\in \mathbb{U}(\bar x)$ and $(x,0) \in H_{\mathcal M}(p)$, by definition, $0 \in {\mathcal M}(p,x) = \mathcal G(p,x) + {gph}(\partial^\pi g)$. As a consequence, $\mathcal G(0,x) - \mathcal G(p,x) \in \mathcal G(0,x) + {gph}(\partial^\pi g)$ and hence $\left(x, \mathcal G(0,x) - \mathcal G(p,x)\right) \in H_{\mathcal M}(0)$. Therefore we have the following  inequality
		\begin{eqnarray*}
			 {\rm{dist}} \left((x,0),H_{\mathcal M}(0)\right) \le \left\| \left(x,0\right) - \left(x, \mathcal G(0,x) - \mathcal G(p,x) \right) \right\|\le \left\| \mathcal G(0,x) - \mathcal G(p,x) \right\| \le L \left\|p\right\|,
		\end{eqnarray*}	which means that $\mathcal M$ has the {\em{restricted calmness property}} with respect to $p$ at $(0,\bar{x},0)$;
		\item We can show that $\mathcal M_0(x) := \mathcal M(0,x)$ is metrically subregular at $(\bar{x}, 0)$ provided that ${\cal S}_{cano}$ is {\em{calm}} at $(0,\bar x)$. Indeed by Proposition \ref{equi_ms}, $\mathcal M_0(x)$ is metrically subregular at $(\bar{x},(0,0))$ if and only if $\nabla f(x) + \partial^\pi g(x)$ is metrically subregular at $(\bar{x},0)$, which is equivalent to the {\em{calmness}} of ${\cal S}_{cano}$ at $(0,\bar x)$.
	\end{itemize}

Conversely suppose that $\mathcal {S}_{PG}$  is {\em{calm}} at $(0,\bar{x})$.
Thanks to the equivalence between {\em{calmness}} of $\mathcal {S}_{PG}$ and $\mathcal {S}_{PGb}$, we have that $\mathcal {S}_{PGb}$ is {\em{calm}} at $(0,\bar{x})$ as well.
By definition,  there exist a neighborhood ${{\mathbb{U}}}(\bar x)$ of $\bar{x}$ and $\delta >0, \kappa >0$ such that
\begin{equation}\label{calmPGb}
{\hbox{dist}} \left( x, {\cal{S}}_{PGb} \left(0\right)\right)  \le  \kappa  \left\|p\right\|,\quad \forall x \in {{\mathbb{U}}}(\bar x),  p\in {\cal{S}}_{PGb}^{-1}(x), \|p\|\le \delta.
\end{equation}
Take any $x \in \mathcal {S}_{cano}(p)$. By definition, $p \in \nabla f(x) + \partial^\pi g (x)$, which  can be rewritten as,
\[\label{219}
\gamma p + x - \gamma\nabla f(x) \in x + \gamma\partial^\pi g (x).
\]
Since  $\nabla f$ is Lipschitz continuous with modulus $L$,  $I-\gamma\nabla
	f$ where $I$ is the identity matrix of size $n$,  is maximally monotone and strongly monotone with constant $1-\gamma L$ if $\gamma < \frac{1}{L}$. Therefore, $(I-\gamma\nabla
	f)^{-1}$ is well defined and Lipschitz continuous with modulus $\frac{1}{1-\gamma L}$. That is, there exists $\tilde{x}
	$ such that
	\begin{eqnarray}\label{titlex-exp}
	\gamma p + x - \gamma\nabla f({x}) =  \tilde{x} - \gamma\nabla f(\tilde{x}),
	\end{eqnarray}
	and consequently
	\begin{equation}
		\left\| \tilde{x} - x \right\| = \left\| \left( I - \gamma\nabla
		f\right)^{-1} \left( \gamma p + x - \gamma\nabla f(x) \right) - \left(I - \gamma\nabla
		f \right)^{-1} \left( x - \gamma\nabla f(x) \right)\right\|\le \frac{\gamma}{1-\gamma L}\left\|p\right\|.\label{upperb}
	\end{equation}
	Plugging (\ref{titlex-exp}) into \eqref{219}, we have $	\frac{\tilde{x} - x}{\gamma} \in \nabla f \left( \tilde{x} \right) + \partial^\pi g \left( \tilde{x} - \tilde{x} + x \right)$, and as a result $\tilde{x} \in {\cal{S}}_{PGb} \left( \tilde{x} - x \right)$. Moreover by (\ref{upperb}),  there exist $ {{\mathbb{U}}}_0(\bar x)$, a neighborhood of $\bar x$ such that
	$$  \tilde{x} \in {\mathbb{U}}(\bar x), \quad \forall x \in {{\mathbb{U}}}_0(\bar x),\left\| p \right\| \le \delta. $$
	To summerize, by (\ref{calmPGb}),  for any $x, p$  with $x \in {\mathbb{U}}_0(\bar x)$, $\|p\| \le \delta$, $p\in {\cal S}_{cano}^{-1}(x)$, we have the estimate
	\begin{eqnarray}
	\hbox{dist} \left( x,\mathcal{S}_{cano}(0)\right) &\le& \hbox{dist} \left(\tilde{x},\mathcal{S}_{cano}(0) \right) + \left\|\tilde{x}-x\right\| = \hbox{dist} \left( \tilde{x},\mathcal{S}_{PGb}(0) \right) + \left\|\tilde{x}-x\right\| \nonumber\\
	& \le& \left(\kappa + 1 \right) \left\| \tilde{x} - x \right\| \le \frac{\left(\kappa + 1 \right)\gamma}{1-\gamma L} \left\|p\right\|.\nonumber
	\end{eqnarray}
Hence the {\em{calmness}} of $\mathcal{S}_{cano}$ at $\left(0,\bar{x}\right)$ follows by definition.
	
{\bf (iii)}:
 We assume that ${\cal{S}}_{cano}$ is {\em{calm}} at $(0,\bar{x})$. Then  there is a neighborhood ${\mathbb{U}}(\bar x)$ of $\bar{x}$ and  $\delta>0, \kappa > 0$ such that
 \begin{equation}\label{calmcano}
	{\hbox{dist}} \left( x, {\cal{S}}_{cano} \left(0\right)\right)  \le  \kappa  \left\|p\right\|,\quad \forall x \in {{\mathbb{U}}}(\bar x),  p\in {\cal{S}}_{cano}^{-1}(x), \|p\|\le \delta.
	\end{equation}
	Let  $x \in {\cal{S}}_{PPA} \left( p \right)$. Then by definition
	$$
	\frac{p}{\gamma} \in \nabla f \left( x - p \right) + \partial^\pi g \left( x - p \right).
	$$Let $\tilde{x} = x - p$, then
	$$
	\frac{p}{\gamma} \in \nabla f (\tilde{x}) + \partial^\pi g (\tilde{x}),
	$$ which implies $\tilde{x} \in {\cal{S}}_{cano} \left(\frac{p}{\gamma}\right)$. Let $ {{\mathbb{U}}}_0(\bar x)$ be a neighborhood of $\bar x$ and $\delta>0$ be such that
	$$  \tilde{x} = x-p \in {\mathbb{U}}(\bar x), \quad \forall x \in {{\mathbb{U}}}_0(\bar x), \left\| p \right\| \le \delta. $$
	Then by (\ref{calmcano}), we have for any $x, p $ with $x\in  {{\mathbb{U}}}_0(\bar x), \|p\|<\delta, p \in {\cal S}_{PPA}^{-1}(x)$, $$
	{\hbox{dist}} \left( x, {\cal{S}}_{PPA} \left(0\right)\right) \le {\hbox{dist}} \left( x - p, {\cal{S}}_{PPA} \left(0\right)\right) + \left\| p \right\| = {\hbox{dist}} \left( \tilde{x}, {\cal{S}}_{cano} \left(0\right)\right) + \left\|p\right\| \le (\kappa/\gamma+1) \|p\|,
	$$
	which means that ${\cal S}_{PPA} $ is {\em{calm}} at $(0,\bar x)$.

	Conversely, we now assume that ${\cal{S}}_{PPA}$ is {\em{calm}} at $(0,\bar{x})$. Then  there is a neighborhood ${\mathbb{U}}(\bar x)$ of $\bar{x}$ and  $\delta>0, \kappa > 0$ such that
 \begin{equation}\label{calmPPA}
	{\hbox{dist}} \left( x, {\cal{S}}_{PPA} \left(0\right)\right)  \le  \kappa  \left\|p\right\|,\quad \forall x \in {{\mathbb{U}}}(\bar x),  p\in {\cal{S}}_{PPA}^{-1}(x), \|p\|\le \delta.
	\end{equation}
Let $x \in {\cal{S}}_{cano} \left(p\right)$. Then  $p \in \nabla f(x) + \partial^\pi g (x)$. Denote by $\tilde{x} = x + p$. Then $ p  \in \nabla f(\tilde{x}-p) + \partial^\pi g (\tilde{x}-p)$. Hence $\tilde{x} \in {\cal{S}}_{PPA} \left(\gamma p\right)$.
Let $ {{\mathbb{U}}}_0(\bar x)$ be a neighborhood of $\bar x$ and $\delta>0$ be such that
	$$  \tilde{x} = x+p \in {\mathbb{U}}(\bar x), \quad \forall x \in {{\mathbb{U}}}_0(\bar x), \left\| p \right\| \le \delta. $$
Then by (\ref{calmPPA}), we have for any $x, p $ with $x\in  {{\mathbb{U}}}_0(\bar x), \|p\|<\delta, p \in {\cal S}_{cano}^{-1}(x)$,
	$$
	{\hbox{dist}} (x,\mathcal {S}_{cano}(0)) \le {\hbox{dist}} \left( x + p, {\cal{S}}_{cano}(0) \right) + \left\|p\right\| = {\hbox{dist}} \left( \tilde{x}, {\cal{S}}_{PPA} (0) \right) + \left\|p\right\| \le \left( \gamma \kappa + 1 \right) \left\|p\right\|.
	$$
Hence ${\cal S}_{cano}$ is {\em{calm}} at $(0,\bar x)$.

{\bf (iv)} Suppose that ${\cal{S}}_{PGb}$ is {\em{calm}}  at $(0,\bar{x})$. Then  there is a neighborhood ${\mathbb{U}}(\bar x)$ of $\bar{x}$ and  $ \kappa > 0$ such that
\begin{equation}\label{4.18}
{\hbox{dist}} \left(x,{\cal{S}}_{PGb}(0)\right) \le \kappa\, \|p\| \quad \forall  x \in {\mathbb{U}}(\bar x) \cap {\cal{S}}_{PGb}(p).
\end{equation}
Let $x \in {\mathbb{U}}(\bar x)$ and any $x^+ \in  \hbox{Prox}^{\gamma}_{g} \left( x - \gamma \nabla f(x) \right) $.
Then by the definition of the proximal operator
and the optimality condition,
	\begin{equation}\label{proximal_subproblem_optimality}
	0 \in \gamma \partial^\pi g (x^+) + x^+ - \left(x - \gamma \nabla f(x) \right),
	\end{equation}
	  Or equivalently
	$$
	\frac{x -  x^+}{\gamma} \in \nabla f(x) + \partial^\pi g \left(x - \left( x -  x^+ \right)\right).
	$$ It follows that
	\begin{equation}\label{S-LT_perturb}
	x \in {\cal{S}}_{PGb} \left( x -  x^+ \right) .
	\end{equation}
Putting $p = x -  x^+$ in  (\ref{4.18}) and noticing that (\ref{S-LT_perturb}) holds and $x^+$ is an arbitrary element in $ \hbox{Prox}^{\gamma}_{g} \left( x - \gamma \nabla f(x) \right) $, we have
$$
	 \text{dist} \left(x,{\cal{X}}^\pi\right) \le \kappa\, {\rm{dist}} (x,  {\rm{Prox}}^{\gamma}_{g} \left( x - \gamma \nabla f(x) \right) ),\quad \forall x \in {\mathbb{U}}(\bar x).
$$ That is, the proximal error bound holds at $\bar x$.
	
	Next let us consider the reverse direction. Suppose there exists $ {\mathbb{U}}(\bar x)$ such that
\begin{equation} \label{proxe}\text{dist} \left(x,{\cal{X}}^\pi\right) \le \kappa\, {\rm{dist}} (x,  {\rm{Prox}}^{\gamma}_{g} \left( x - \gamma \nabla f(x) \right) ),\quad \forall x \in {\mathbb{U}}(\bar x).
\end{equation}  Without loss of generality, assume that  $g$ is a semi-convex function on $\mathbb{U}(\bar x)$ with modulus $\rho$.  Then when $\gamma \le \frac{1}{\rho}$, the function $g(x)+\frac{1}{2\gamma}\|x\|^2$ is convex on $\mathbb{U}(\bar x)$ and consequently $g(x)+\frac{1}{2\gamma}\|x-a\|^2$ is convex on $\mathbb{U}(\bar x)$ for any $a$.
	{Let $ {{\mathbb{U}}}_0(\bar x)$ be a neighborhood of $\bar x$ and $\delta>0$ be such that
	$  x-p \in {\mathbb{U}}(\bar x), \ \forall x \in {{\mathbb{U}}}_0(\bar x), \left\| p \right\| \le \delta. $
	Then by the optimality condition and the convexity of $x'\rightarrow g(x')+\frac{1}{2\gamma}\|x'-x+\gamma \nabla f(x)\|^2$, for any $x \in {{\mathbb{U}}}_0(\bar x), \left\| p \right\| \le \delta$,
	$$
	x-p \in  \hbox{Prox}^{\gamma}_{g} \left( x - \gamma \nabla f(x) \right) \quad\Longleftrightarrow \quad \frac{p}{\gamma} \in \nabla f(x) + \partial^\pi g \left(x - p\right).$$
It follows from  the proximal error bound (\ref{proxe}) that 	
$$ \text{dist} \left(x,{\cal{X}}^\pi\right) \le \kappa \|x-p\|, \quad \forall x \in {\mathbb{U}}_0(\bar x),  x\in {\cal{S}}_{PGb}(p), \|p\|\leq \delta,$$ i.e., ${\cal S}_{PGb}$ is calm at $(0,\bar x)$.}

	{\bf (v)}: The claim follows straightforwardly by the continuity of $F(x)$ in the domain.

{\bf (vi)}: We now suppose that $g$ is  semi-convex around $\bar x$ with modulus $\rho$.
It directly follows from {(i)(ii)(iv)} that if the proximal error bound \eqref{Luo-Tseng-neighborhood-error-bound} at $\bar{x}$ holds with $\gamma < \min\{1/\rho,1/L\}$, then the calmness of $\mathcal {S}_{cano}$ at $(0,\bar{x})$ holds, i.e., there exist $\kappa , \epsilon_1 >0$ such that
\begin{equation}\label{vi_eq1}
{\rm{dist}} \left(x, {\cal{X}}^\pi\right) \le \kappa {\rm{dist}}(0, \partial^\pi F(x)),\ \forall x \in {\mathbb{B}} \left(\bar{x},\epsilon_1\right).
\end{equation}
Moreover, thanks to the proper separation of stationary value (\ref{PBisocost}), there exists $\epsilon_2 > 0 $ such that
\begin{equation}\label{sr}
{\cal{X}}^\pi \cap {\mathbb{B}} \left(\bar{x},\epsilon_2\right) \subseteq \{x ~|~ F(x) = F(\bar{x})\}.
\end{equation}
Let $\epsilon:= \min\{\epsilon_1,\epsilon_2\}$. Inspired by the technique in \cite[Proposition 3.8]{drusvyatskiy2016nonsmooth}, without loss of generality, assume that $\nabla f$ is Lipschitz continuous with modulus $L$ and $g$ is  semi-convex around  with modulus $\rho$ on $ {\mathbb{B}} \left(\bar{x},\epsilon \right)$.
 Let $f_\tau(x):= f(x) + \frac{\tau}{2} \|x\|^2$. Since   $\nabla f$ is Lipschitz continuous with modulus $L$, for any $\tau \ge L$, we have
$$
\begin{aligned}
\langle \nabla f_\tau(x_1) - \nabla f_\tau(x_2),x_1 - x_2 \rangle &= \langle \nabla f(x_1) - \nabla f(x_2),x_1 - x_2 \rangle + \tau\|x_1- x_2 \|^ 2 \\
& \ge -L\|x_1- x_2 \|^ 2 +  \tau\|x_1- x_2 \|^ 2 \ge 0, \qquad \forall x_1,x_2 \in dom\,f\cap {\mathbb{B}} \left(\bar{x},\epsilon \right),
\end{aligned}
$$
implying $f_\tau$ is convex on $ {\mathbb{B}} \left(\bar{x},\epsilon \right)$ and thus $f$ is semi-convex  at $\bar x$. It follows that  $F = f+g$ is semi-convex around $\bar x$ on since $g$ is semi-convex around $\bar x$ as well. {Without loss of generality,  we assume that $F$ is a semi-convex function on ${\mathbb{B}} \left(\bar{x},\epsilon \right)$. Then for $x \in {\mathbb{B}} \left(\bar{x},\epsilon \right)$ , we have $\partial^{\pi} F(x) = \partial F(x)$ by Proposition \ref{Prop2.5}. Moreover since $F$ is a sum of a convex function and a positive multiple of the function  $\|x\|^2$, it is easy to verify that  there exists $C \ge 0$ such that for any $x_1,x_2 \in dom\,F \cap {\mathbb{B}} \left(\bar{x},\epsilon \right)$ and $\xi \in \partial^{\pi} F(x_1)$
$$
F(x_2) \ge F(x_1) + \langle \xi, x_2 - x_1\rangle - C\|x_2-x_1\|^2.
$$}
Now, given any $x \in {\mathbb{B}} \left(\bar{x},\epsilon/2 \right)$, since $\partial^{\pi} F(x) = \partial F(x)$, $\mathcal{X}^{\pi}$ is closed and and hence we can find  $x_0$,  the projection of $x$ on $\mathcal{X}^{\pi}$, i.e., $\|x_0 - x\| = {\rm{dist}}(x, \mathcal{X}^{\pi})$. Then
$$
\| x_0 - \bar{x}\| \le \|x_0 - x\| + \|x - \bar{x}\| \le 2\|x - \bar{x}\| \le \epsilon,
$$
and thus by (\ref{sr}),
$
F(x_0) = F(\bar{x}).
$
Therefore, for any $\xi \in \partial^{\pi} F(x)$, we have
$$
F(\bar{x})= F(x_0) \ge F(x) + \langle \xi, x_0 - x\rangle - C\|x_0-x\|^2 \ge F(x) - \| \xi\| \| x_0 - x\| - C\|x_0-x\|^2.
$$
By the arbitrariness of $\xi$, we have
$$
F(\bar{x}) \ge F(x) - {\rm{dist}}^2(0, \partial^{\pi} F(x)){\rm{dist}}(x, \mathcal{X}^{\pi}) - C{\rm{dist}}^2(x, \mathcal{X}^{\pi}).
$$
Combing with \eqref{vi_eq1}, we get
$$
F(x) - F(\bar{x}) \le (C+\kappa){\rm{dist}}^2(0, \partial^{\pi} F(x)),\ \ \forall x \in {\mathbb{B}} \left(\bar{x},\frac{\epsilon}{2} \right).
$$
\end{proof}

\subsection{Calmness of ${\cal S}_{cano}$ in convergence analysis and KL exponent calculus}
We next discuss applications of the calmness of ${\cal S}_{cano}$ in convergence analysis and KL exponent calculus, justifying the promised (R2) and (R3) in Section 1.
As discussed at the beginning of this section, the calmness of ${\cal S}_{PG}$ at $(0,\bar x)$  is a sufficient condition  for the PG-iteration-based error bound to hold at $\bar x$. Hence by Theorems \ref{Thm3.2} and  \ref{Full_theorem}, we obtain the following linear convergence result of the PG method for nonconvex problems.
\begin{theorem}\label{linearcon-calm}
 Assume $\gamma<\frac{1}{L}$. Let the sequence $\{x^k\}$ be generated by the PG method with $\bar x$ as an accumulation point. Then  $\bar x\in {{\cal X}}^\pi$ and the sequence $\left\{ x^k \right\}$ converges to $\bar x $ linearly if the proper separation of stationary value (\ref{PBisocostassp}) holds at $\bar x$ and the set-valued map ${\cal S}_{cano}$ is calm at $(0,\bar x)$.
\end{theorem}

In \cite[Theorem 4.1]{LiPong2016calculus}, assuming that $g$ is convex, the authors show that Luo-Tseng error bound (\ref{Luo-Tseng-error-bound}) together with the proper separation of stationary value (\ref{isocost}) are sufficient for the KL property of $F$ with exponent $\frac{1}{2}$. Theorem \ref{Full_theorem} inspires a weaker sufficient condition for the KL property as follows.
\begin{theorem}\label{KLsufficient}
Let $\bar x\in {{\cal X}}^\pi$. Suppose that $g$ is  semi-convex around $\bar x$,  the proper separation of stationary value (\ref{PBisocost}) at $\bar{x}$ holds and the set-valued map ${\cal S}_{cano}$ is calm at $(0,\bar x)$. Then $F$ has the KL property at $\bar{x}$ with an exponent of $\frac{1}{2}$.
\end{theorem}

\section{Verification of the calmness of $\mathcal {S}_{cano}$}\label{Sec5}
\setcounter{equation}{0}

Based on recent developments in variational analysis, there are more and more  sufficient conditions for  verifying  calmness of ${\cal S}_{cano}$ available. In this section, we will summarize some of these conditions and demonstrate how we could verify the desired calmness condition, both for structured convex problems and general nonconvex problems.  Moreover for the structured convex case, we shall discuss the advantage of using the calmness of ${\cal S}_{cano}$ instead of the Luo-Tseng error bound in both the convergence analysis and KL exponent calculus.

 \subsection{The calmness of $\mathcal {S}_{cano}$ for structured convex problems}

It is known that under  the structured convex assumption  \ref{assum_struc}, the solution set
can be rewritten as
${\cal X}=\{x| 0=Ax-\bar y, 0=\bar{g}+\partial g(x)\},$ where $\bar y, \bar g$ are some constants; see e.g.,  \cite[Lemma 4.1]{FBconvex}.
It then follows that under the structured convex assumption  \ref{assum_struc}, the calmness of $\mathcal {S}_{cano}$ is equivalent to the calmness of the following perturbed solution map:
$\Gamma (p_1,p_2)=\{x| p_1=Ax-\bar y, p_2=\bar{g}+\partial g(x)\}$; see e.g.,  \cite[Proposition 4.1]{FBconvex}.

Using the calmness intersection theorem (see {\cite[Theorem 3.6]{KK-2002}}), if $\partial g$ is metrically subregular at $(\bar x, -\bar g)$ and the set $\{x|0\in \bar g+\partial g (x)\}$ is a convex polyhedral set, then according to Proposition \ref{sufficient-MSCQ-new}(1), $\Gamma (p_1,p_2)$ is calm at $(0,0,\bar x)$. By using this technique,
  \cite[Theorem 4.4]{FBconvex} has shown that under the structured convex assumption \ref{assum_struc},  if $g$ is the group LASSO regularizer, then the set-valued map ${\cal S}_{cano}$ is calm at $(0,\bar x)$. It follows from Theorem \ref{linearcon-calm} that we have the following linear convergence result.
\begin{proposition}\label{Grouplasso}
 Consider the convex optimization problem (\ref{Basic_Problem}) where $f$ satisfies Assumption \ref{assum_struc} and $g$ is the group LASSO regularizer.
 Assume $\gamma<\frac{1}{L}$.
 Let the sequence $\{x^k\}$ be generated by the PG method. Then the sequence $\left\{ x^k \right\}$ converges to an optimal solution $\bar x $ linearly.
\end{proposition}
Recall that the classical result of linear convergence of the PG method under  (C3)  was shown by using  the Luo-Tseng error bound. By Theorem \ref{Full_theorem}, when $g$ is convex, the Luo-Tseng error bound is in general stronger than all calmness conditions unless the set of stationary points are compact. Hence using the point-based calmness condition instead of the Luo-Tsend error bound, the above result improves the classical result by removing the compactness assumption on the solution set.   This example demonstrates the advantage of using the point-based calmness condition over the Luo-Tseng error bound condition and justifies (R5) we have promised in Section 1.

Now consider the structured convex case where $g$ is a convex piecewise linear-quadratic (PLQ) function. In this case  $\partial g$ is a polyhedral multifunction, i.e., its graph is the union of finitely many polyhedral convex sets. Hence by Robinson's polyhedral multifunction theory, $\Gamma$ is a polyhedral multifunction and hence upper-Lipschitz continuous. Consequently, $\Gamma$ is calm at $(0,0,\bar x)$ and therefore
the set-valued map ${\cal S}_{cano}$ is calm at  $(0,\bar x)$ for any solution $\bar x $ of (\ref{Basic_Problem}). It follows from  Theorem \ref{Full_theorem} that we have the following result which has improved the result obtained in \cite[Proposition 4.1]{LiPong2016calculus}  by eliminating the compactness assumption of the solution set.
\begin{proposition}\label{PLQ-KL}
 Consider convex optimization problem (\ref{Basic_Problem}). If $f$ satisfies Assumption \ref{assum_struc} and $g$ is a convex piecewise linear-quadratic (PLQ) function, then $F$ has the KL property at $\bar{x}$ with an exponent of $\frac{1}{2}$.
\end{proposition}
Again the improvement is due to the replacement of the Luo-Tseng error bound by the point-based calmness condition. Moreover this example justifies (R3) we have promised in Section 1.

   \subsection{The calmness of $\mathcal {S}_{cano}$ for general nonconvex problems}

  We next develop some verifiable sufficient conditions for the calmness in terms of the problem data by using  Proposition \ref{sufficient-MSCQ-new}(1-3); and then illustrate them by some concrete applications popularly appearing in statistical learning fields.
\begin{theorem}\label{polymulti-thm} Let $\bar{x}\in {\cal{X}}^\pi$ { and suppose that ${gph \left(\partial^\pi g \right)}$ is closed near the point $\bar x$}. Then  the set-valued map  $\mathcal {S}_{cano}$ is calm at $(0, \bar x)$ if one of the following conditions holds.
\begin{itemize}
\item[1.]The mapping $\nabla f$ is piecewise affine and $\partial^\pi g$ is a polyhedral multifunction.
\item[2.]NNAMCQ holds at $\bar x$:
	$$ 0\in \xi+\partial \langle -\nabla f, \eta\rangle (\bar x),\quad (\xi,\eta) \in {\cal N}_{gph \left(\partial^\pi g \right)}(\bar x, -\nabla f(\bar x)) \Longrightarrow\ (\xi,\eta)=0.$$
	When $g$ is separable, i.e., $g(x) = \sum_{i=1}^{n} g_i(x_i)$ and $f$ is twice continuoulsy differentiable,  NNAMCQ holds at $\bar x$ if
\begin{eqnarray*}
&&  0= \xi_i-\nabla f_{x_i}
(\bar x)^T \eta, ~\quad
 ~ (\xi_i,\eta_i) \in {\cal N}_{gph \left(\partial^\pi g_i \right)}(\bar x_i, -f_{x_i}(\bar x)) , \quad i=1,\dots, n \\
 &&  \Longrightarrow\ (\xi,\eta)=0.
\end{eqnarray*}
	\item[3.]  FOSCMS holds at $\bar{x}$: $f$ is twice continuoulsy differentiable and for every $w\not =0$ such that
	$$ (w,-\nabla^2 f(\bar x)w) \in {\cal T}_{gph ( \partial^\pi g)}(\bar x, -\nabla f(\bar x)) $$ one has
	$$
	0= \xi-\nabla^2 f(\bar{x})^T \eta,\quad  (\xi,\eta)\in {\cal N}_{gph \left(\partial^\pi g\right)}((\bar x, -\nabla f(\bar x)); (w, -\nabla^2 f(\bar x)w))\Longrightarrow \  (\xi,\eta) =0 .
	$$
When $g$ is separable, FOSCMS holds at $\bar{x}$ provided that  $f$ is twice continuoulsy differentiable and for every
$0 \neq w $ such that $$ ~(w_i,-  \nabla f_{x_i}(\bar{x})^Tw) \in {\cal T}_{gph ( \partial^\pi g_i)}(\bar{x}_i, -f_{x_i}(\bar x)),\quad i=1,\dots, n$$
one has
\begin{eqnarray*}
&& 0= \xi_i-\nabla f_{x_i}
(\bar x)^T \eta, \quad  \begin{array}{l}
 ~ (\xi_i,\eta_i) \in {\cal N}_{gph \left(\partial^\pi g_i \right)}((\bar x_i, -f_{x_i}(\bar x));(w_i, -\nabla f_{x_i}(\bar x)^Tw))\end{array} , \quad i=1,\dots, n \\
 &&  \Longrightarrow\ (\xi,\eta)=0.
\end{eqnarray*}

\end{itemize}
\end{theorem}
\begin{proof}
Since
$$
 {\cal{X}}^\pi := \left\{ x \,\big|\, 0\in \nabla f(x)+\partial^\pi g(x) \right\}=\left\{x\, \big|\, (x,-\nabla f(x)) \in gph ( \partial^\pi g)\right\}, $$
by Proposition \ref{equi_ms}, the set-valued map  $M_1(x):=\nabla f(x)+\partial^\pi g(x)$ is metrically subregular
at $(\bar x, 0)$ if and only if the set-valued map $M_2(x):=(-x, \nabla f(x))+gph (\partial^\pi g)$ is metrically subregular at $( \bar x, 0,0)$. For the nonseparable case, the results follow from Proposition \ref{sufficient-MSCQ-new}(1-3) and the nonsmooth calculus rule in Proposition \ref{calculus} by taking $P(x):= (x, -\nabla f(x))$ and $D:=gph (\partial^\pi g)$.
For the separable case, $ x\in {\cal X}^\pi$ if and only if
$$ 0\in f_{x_i}( x)+\partial^\pi g_i(x_i) \quad i=1,2,\dots,n.$$
Equivalently $ x\in {\cal X}^\pi$ if and only if
$$ (x_i,- f_{x_i}( x))\in gph (\partial^\pi g_i) \quad i=1,2,\dots,n.$$ Denote by
$$P_i(x):=(x_i,- f_{x_i}( x)), \quad D_i:=gph (\partial^\pi g_i).$$ Then $x\in {\cal X}^\pi$ if and only if
$$P_i(x) \in D_i, \quad i=1,2,\dots, n.$$
Let $P(x):=(P_1(x),\dots, P_n(x))$ and $D:=D_1\times \dots D_n$. Applying Proposition \ref{sufficient-MSCQ-new}(2-3) and Lemma \ref{lemma3.5}, we obtain the desired results for the separable case.
\end{proof}

We list in Table \ref{Cases2ofinterests} four scenarios of our interest for which the calmness conditions can be verified according to Theorem \ref{polymulti-thm}(2-3). For these four cases, $f$ is chosen from two popular loss functions, e.g., logistic loss and exponential loss, while $g$ is chosen from SCAD and MCP. The definition of SCAD penalty is defined as follows, see, e.g.,  \cite{Fan2001SCAD},
$$
\hbox{SCAD}:= \sum_{i=1}^{n} \phi(x_i),\quad \hbox{where}\quad \phi (\theta):= 	\begin{cases} \lambda|\theta|, & |\theta|\leq \lambda, \\
\frac{-\theta^2+2a\lambda|\theta|-\lambda^2}{2(a-1)}, & \lambda<|\theta|\leq a\lambda,\\
\frac{(a+1)\lambda^2}{2}, & |\theta|> a\lambda,
\end{cases}
$$with $\theta \in \bbR$, $a>2$ and $\lambda>0$. Straightforward calculation  reveals that $\phi$ is proximally regular and
$$
\partial \phi(\theta)=\partial^\pi \phi(\theta) =
\left\{\begin{array}{ll}
0, & \quad \theta < -a\lambda,\\
-\frac{1}{a-1}\theta - \frac{a\lambda}{a-1}, & \quad  - a\lambda\leq\theta < -\lambda, \\
-\lambda, & \quad  -\lambda \leq \theta < 0, \\
\left[ -\lambda, \lambda \right], & \quad  \theta = 0, \\
\lambda, & \quad  0 < \theta\leq \lambda, \\
-\frac{1}{a-1}\theta+\frac{a\lambda}{a-1}, & \quad  \lambda<\theta\leq a\lambda, \\
0, & \quad \theta> a\lambda.
\end{array}\right.
$$
The graph of $\partial  \phi(\theta)$ is marked in bold in Figure \ref{graphSCAD}.
\begin{figure}[H]
	\qquad \qquad \qquad \qquad \qquad \quad \quad \begin{tikzpicture}
	\draw[eaxis] (-3.5,0) -- (3.5,0) node[below] {$\theta$};
	\draw[eaxis] (0,-2) -- (0,2) node[above] {$\partial \phi (\theta)$};
	\draw[elegant,black,domain=-3.5:-\numlambda*\numa] plot(\x,{0});
	\draw[elegant,black,domain=-\numlambda*\numa:-\numlambda] plot(\x,{(-\x - \numa*\numlambda)/(\numa - 1)});
	\draw[elegant,black,domain=-\numlambda:0] plot(\x,{-\numlambda});
	\draw[black,line width=1.2pt] (0,-\numlambda) -- (0,\numlambda) {};
	\draw[dotted,thick,black] (-\numlambda,0) --(-\numlambda,-\numlambda) {};
	\draw[dotted,thick,black] (\numlambda,0) --(\numlambda,\numlambda) {};
	\draw[elegant,black,domain=0:\numlambda] plot(\x,{\numlambda});
	\draw[elegant,black,domain=\numlambda:\numlambda*\numa] plot(\x,{(-\x + \numa*\numlambda)/(\numa - 1)});
	\draw[elegant,black,domain=\numlambda*\numa:3.5] plot(\x,{0});
	\draw[dashed, black, line width=1pt] (0,\numlambda) node[left]{$\lambda$};
	\draw[dashed, black, line width=1pt] (0,-\numlambda) node[right]{$-\lambda$};
	\draw[dashed, black, line width=1pt] (\numlambda*\numa,0) node[below]{$a\lambda$};
	\draw[dashed, black, line width=1pt] (-\numlambda*\numa,0) node[above]{$-a\lambda$};
	\draw[dashed, black, line width=1pt] (\numlambda,0) node[below]{$\lambda$};
	\draw[dashed, black, line width=1pt] (-\numlambda,0) node[above]{$-\lambda$};
	\end{tikzpicture}\caption{Graph of the limiting subdifferential of SCAD penalty}\label{graphSCAD}
\end{figure}
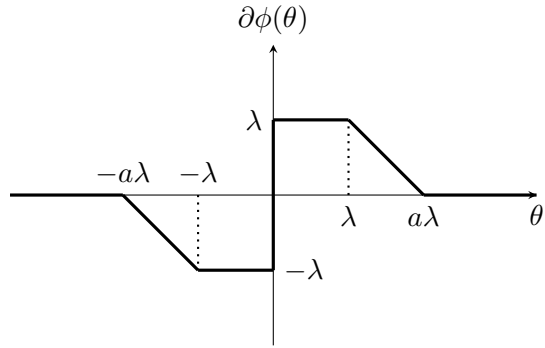

The definition of MCP penalty is as follows, see, e.g.,  \cite{Zhang2010MCP},
$$
\hbox{MCP}:=\sum_{i=1}^n \psi(x_i),\quad \hbox{where}\quad \psi(\theta):= \begin{cases}
\lambda|\theta| - \frac{\theta^2}{2 a}, & |\theta|\leq a \lambda, \\
\frac{a \lambda^2}{2}, & |\theta| > a \lambda,
\end{cases}
$$with $\theta \in \bbR$, $a>1$  and $\lambda>0$. In Table \ref{Cases2ofinterests}, we assume that $x, c_i\in \mathbb{R}^n,  d_i\in \mathbb{R}$.
\begin{table}[!h]
	\caption{Practical scenarios set II leading to the {\em{calmness}} of ${\cal{S}}_{cano}$ under conditions}\label{Cases2ofinterests}
	\begin{center}
		\begin{tabular}{c|c|c|c|c}
			\hline
			$\quad$Scenarios$\quad$ & $\quad\qquad$Case 5$\quad\qquad$ & $\quad$Case 6$\quad$ & $\qquad\quad$Case 7$\quad\qquad$ & $\quad$Case 8$\quad$ \\
			\hline
			$f(x)$ & $ \sum\limits_{i =1}^N - \log \left(1 + e^{ d_i c_i^Tx}\right)$ & $\sum\limits_{i =1}^N e^{ - d_i c_i^T x}$ & $ \sum\limits_{i =1}^N - \log \left(1 + e^{d_i c_i^Tx}\right)$ & $\sum\limits_{i =1}^N e^{ -d_i c_i^T x}$ \\
			\hline
			$g(x)$ & SCAD & SCAD & MCP & MCP \\
			\hline
			\end{tabular}
	\end{center}
\end{table}

We next illustrate how Theorem \ref{polymulti-thm}(2-3) can be applied to verify the {\em{calmness}} of ${\cal{S}}_{cano}$ pointwisely of the four cases in Table \ref{Cases2ofinterests}. For simplicity we focus on   the
case where $f(z)=e^{ - b^T z}$ with $z,b\in \mathbb{R}^n$ and $g(z)=\sum_{i = 1}^n \phi (z_i)$ is a SCAD penalty.
The same technique can be applied to all the cases listed in Table \ref{Cases2ofinterests}.
For simplicity, our discussion is based on a two-dimensional case.
We discuss three kinds of points at which the first point satisfies the strong metric subregularity (isolated calmness), the second point satisfies the metric subregularity (calmness) and the third point satisfies the metric regularity (pseudo Lipschitz continuity).

\begin{example}
 Consider problem (\ref{Basic_Problem}) with $f(z)=e^{-b_1z_1-b_2z_2}$ and $g(z)=\phi(z_1)+\phi(z_2)$.
By straightforward calculation, we have
$$\nabla f({z}) =
		-e^{-b^T z} b ,\,\, \nabla f_{z_i}(z)=e^{-b^T z} b_ib, \
		  \nabla f_{z_i}(z)^T w =e^{-b^T z} b_ib^Tw.
$$
\smallskip
\noindent {\bf Case (i)}:
$\bar{z}_1 = 0,\, -\lambda < e^{ -b^T \bar{z}}\,b_1 < \lambda$ and $\bar{z}_2 = 0,\,  e^{ -b^T \bar{z}}\,b_2 = \lambda$. In this case,  {from  Figure \ref{graphSCAD} it is easy to see  that
$$
 {\mathcal{T}}_{{gph} \left(\partial^\pi  \phi\right)} \begin{pmatrix}
		\bar{z}_1 ,
		e^{ -b^T \bar{z}}\,b_1
		\end{pmatrix}=\{0\}\times \mathbb{R}
, \quad \quad {\mathcal{T}}_{{gph} \left(\partial  \phi\right)}{{ \begin{pmatrix}
			\bar{z}_2, \
			e^{ -b^T \bar{z}}\,b_2
			\end{pmatrix}}}=
			(\mathbb{R}_-\times \{0\}) \cup \{0\}\times  \mathbb{R}_+).
$$
It follows that $w \neq 0$ such that
$~(w_i,- \nabla f_{z_i}(\bar{z})^Tw) \in {\cal T}_{gph ( \partial  \phi)}(\bar{z}_i, -f_{z_i}(\bar z)),$
if and only if $w \neq 0$ and
\begin{equation}
\label{neweq}
(w_1,
			- e^{-b^T \bar{z}} b_1b^Tw) \in \{0\}\times \mathbb{R},
			 \qquad(w_2,
			- e^{-b^T \bar{z}} b_2b^Tw)
		\in (\mathbb{R}_-\times \{0\}) \cup \{0\}\times  \mathbb{R}_+).\end{equation} However if  $w \neq 0$ and (\ref{neweq}) hold, then  $w_1 = 0$, $w_2 \neq 0$ and
$(
			w_2,
			- e^{-b^T \bar{z}} b_2b^Tw
			) \in \mathbb{R}_-\times \{0\}.
$ But this is impossible since $e^{ -b^T \bar{z}}\,b_2 = \lambda > 0$ implies  $b^T w = 0$ which means $w_2 = 0$.}
Consequently, $\mathcal {T}^{lin}_{\tilde{\cal{X}}} \left(\bar{z}\right) = \{0\}$, which means that
FOSCMS holds. In fact in this case,  the set-valued map ${\cal S}_{cano}$ is actually {isolated calm at $(0,\bar z)$}.

\medskip
\noindent {\bf Case (ii)}: $\bar{z}_1 = 0,\,  e^{ -b^T \bar{z}}\,b_1 = -\lambda$ and $\lambda < \bar{z}_2 < a  \lambda,\,  e^{ -b^T \bar{z}}\,b_2 = \lambda+\frac{\bar{z}_2 - \lambda}{1 -a}$ and $b_2 \neq -1/(\bar{z}_2 - a\lambda)$.  In this case, it is easy to see from  Figure \ref{graphSCAD} that
$$ {\mathcal{T}}_{{gph} \left(\partial  \phi\right)}
		\begin{pmatrix}
			\bar{z}_1 ,~
			e^{ -b^T \bar{z}}\,b_1
			\end{pmatrix}= (\mathbb{R}_-\times \{0\}) \cup \{0\}\times  \mathbb{R}_+) , \quad \quad {\mathcal{T}}_{{gph} \left(\partial  \phi\right)}
			\begin{pmatrix}
				\bar{z}_2 ,~
				e^{ -b^T \bar{z}}\,b_2
				\end{pmatrix} =\left \{ t(1,
											\frac{1}{1- a}
											):t\in \mathbb{R}\right \}.
$$
It follows that  $\left(w_i,-\nabla  f_{z_i}(\bar z)^Tw\right) \in {\cal T}_{gph ( \partial  \phi)}(\bar z_i, - f_{z_i}(\bar z)), i = 1,2$
if and only if
$$(
w_1,
- e^{-b^T \bar{z}} b_1b^Tw
) \in (\mathbb{R}_-\times \{0\}) \cup \{0\}\times  \mathbb{R}_+) \qquad (
w_2 ,
- e^{-b^T \bar{z}} b_2b^Tw
) =t(1,
\frac{1}{1- a})
.
$$
for some $t\in \mathbb{R}$.
Because $e^{ -b^T \bar{z}}\,b_1 < 0$, if $(
w_1,
- e^{-b^T \bar{z}} b_1b^Tw
) \in (\mathbb{R}_-\times \{0\})$,  then $b^Tw = 0$ and hence $t = 0$, $w_2 =0$.
Moreover, since $b^Tw = 0, b_1 \neq 0$, then $w_1 = 0$ as well.
Therefore { $(w_1,w_2) \neq 0$ such that
$\left(w_i,-\nabla  f_{z_i}(\bar z)^Tw\right) \in {\cal T}_{gph ( \partial  \phi)}(\bar z_i, - f_{z_i}(\bar z)), i = 1,2$
if and only if  $(w_1,w_2) \neq 0$ such that $$(
w_1,
- e^{-b^T \bar{z}} b_1b^Tw)
\in \{0\}\times\mathbb{R}_+, \qquad  (w_2,- e^{-b^T \bar{z}} b_2b^Tw
) =t(1,
\frac{1}{1- a}),
$$
for some $t\in \mathbb{R}\backslash \{0\}$, if and only if
$$(w_1,- e^{-b^T \bar{z}} b_1b^Tw) =(0, t_1), t_1 \not =0 \quad  (w_2,- e^{-b^T \bar{z}} b_2b^Tw
) =t_2(1,
\frac{1}{1- a}), t_2\not =0.
$$
From  Figure \ref{graphSCAD}, since $ \bar{z}_1=0,
e^{ -b^T \bar{z}}b_1=-\lambda$, we have
{\begin{equation}  {\mathcal{N}}_{{gph} \left(\partial  \phi\right)} \left(( \bar{z}_1,
					e^{ -b^T \bar{z}}b_1);(0,1)
									\right)=
									\mathbb{R}\times \{0\}.\label{normalcone1}
\end{equation}} Since $\lambda < \bar{z}_2 < a  \lambda,\,  e^{ -b^T \bar{z}}\,b_2 = \lambda+\frac{\bar{z}_2 - \lambda}{1 -a}$, we have
\begin{equation}  {\mathcal{N}}_{{gph} \left(\partial  \phi\right)} \left ((\bar{z}_2,
					e^{ -b^T \bar{z}}b_2)
					; (1,	\frac{1}{1- a}
									)\right) =\left \{ u(1,a-1): u\in \mathbb{R}\right  \}.\label{normalcone2}
\end{equation}
Suppose that
\begin{eqnarray*}
&&(w_1,w_2) \neq 0,\quad
\left(w_i,-\nabla  f_{z_i}(\bar z)^Tw\right) \in {\cal T}_{gph ( \partial  \phi)}(\bar z_i, - f_{z_i}(\bar z)), i = 1,2,\\
&& 0= \xi_i- \nabla f_{z_i}
 (\bar z)^T\eta, ~ (\xi_i,\eta_i) \in {\cal N}_{gph \left(\partial  \phi \right)}((\bar z_i, -f_{z_i}(\bar z));\left(w_i,-\nabla  f_{z_i}(\bar z)^Tw\right)) , \quad i=1,2 .
\end{eqnarray*}
Then 	from the analysis above,			
\begin{eqnarray}
&& (w_1,- e^{-b^T \bar{z}} b_1b^Tw) =(0, t_1), t_1 \not =0 \quad  (w_2,- e^{-b^T \bar{z}} b_2b^Tw
) =t_2(1,
\frac{1}{1- a}), t_2\not =0 \nonumber \\
&&0=\xi_1 - e^{ -b^T \bar{z}}\,b_1 (b_1\eta_1+b_2\eta_2)  , \quad 0=\xi_2 - e^{ -b^T \bar{z}}\,b_2 (b_1\eta_1+b_2\eta_2) ,\label{CQchecking1}\\
&& (\xi_1,\eta_1) \in \mathbb{R}\times \{0\}, \quad (\xi_2,\eta_2) \in \left \{ u(1,a-1): u\in \mathbb{R}\right  \},\nonumber
\end{eqnarray}
 where the last inclusions follow by (\ref{normalcone1}) and (\ref{normalcone2}).  It follows that $\eta_1 = 0, \xi_2=u, \eta_2=u(a-1)$.
 Hence (\ref{CQchecking1}) becomes
$$
0=\xi_1 - e^{ -b^T \bar{z}}\,b_1 b_2u(a-1), \quad 0=u - e^{ -b^T \bar{z}}\,b_2^2u(a-1) .
$$}
Form the second equality, we have
$$
\begin{aligned}
u - b_2u(a - 1)e^{ -b^T \bar{z}}\,b_2 &= u\left(1-b_2(a-1)(\lambda+\frac{\bar{z}_2 - \lambda}{1 -a})\right) \\
& = u\left(1+b_2(\bar{z}_2 - a\lambda)\right) = 0,
\end{aligned}
$$
which implies that $u=0$ provided $b_2 \neq -1/(\bar{z}_2 - a\lambda)$ and hence that $\xi_1 = 0$.
It follows that $(\xi,\eta)=0$.
Hence FOSCMS holds at $\bar z$ which implies that the set-valued map ${\cal S}_{cano}$ is calm around $(0,\bar z)$.

\medskip
\noindent {\bf Case (iii)}: $\bar{z}_1 = 0,\,  -\lambda < e^{ -b^T \bar{z}}\,b_1 < \lambda$ and $0 < \bar{z}_2 < \lambda,\,  e^{ -b^T \bar{z}}\,b_2 = \lambda$.
We first calculate the limiting normal cone to $ {gph} \left(\partial \phi \right)$ at $(\bar{z}_i ,~
e^{ -b^T \bar{z}}\,b_i)$.
From  Figure \ref{graphSCAD}, since $ \bar{z}_1=0,
-\lambda < e^{ -b^T \bar{z}}\,b_1 < \lambda$, we have
\begin{equation}  {\mathcal{N}}_{{gph} \left(\partial \phi\right)} \left( \bar{z}_1,
e^{ -b^T \bar{z}}b_1
\right)=\mathbb{R}\times
\{0\}. \label{normalcone_1}
\end{equation} Since $0 < \bar{z}_2 < \lambda,\,  e^{ -b^T \bar{z}}\,b_2 = \lambda$, we have
\begin{equation}  {\mathcal{N}}_{{gph} \left(\partial  \phi\right)} \left(\bar{z}_2,
e^{ -b^T \bar{z}}b_2
\right) =\{0\}\times \mathbb{R}.\label{normalcone_2}
\end{equation}
Suppose that
\begin{eqnarray*}
&& 0= \xi_i- \nabla f_{z_i}
 (\bar x)^T\eta, ~ (\xi_i,\eta_i) \in {\cal N}_{gph \left(\partial  \phi \right)}(\bar z_i, -f_{z_i}(\bar z)), \quad i=1,2 .
\end{eqnarray*}
It follows that						
\begin{eqnarray*}
0=\xi_1 - e^{ -b^T \bar{z}}\,b_1 (b_1\eta_1+b_2\eta_2)  , \quad \xi_2 - e^{ -b^T \bar{z}}\,b_2 (b_1\eta_1+b_2\eta_2)  = 0.
\end{eqnarray*}
Then by (\ref{normalcone_1})- (\ref{normalcone_2}), since $\eta_1=0, \xi_2=0$, it follows that
$$
0=\xi_1 - e^{ -b^T \bar{z}}\,b_1b_2 \eta_2, \quad - b_2\eta_2e^{ -b^T \bar{z}}\,b_2 = 0 .
$$
which implies $\xi_1=\eta_2 = 0$.
Hence NNAMCQ holds at $\bar z$. So in this case,  the set-valued map ${\cal S}_{cano}$ is actually pseudo-Lipschitz around $(0,\bar z)$.
\end{example}

\section{Application of the perturbation analysis technique}

Our last goal is to provide an answer to question \textbf{Q}. Taking the ADMM and PDHG for illustration, the new perturbation analysis technique determines an appropriately perturbed stationary point set-valued map and hence the calmness condition tailored to the algorithm under investigation.

\subsection{Error bound and linear convergence of ADMM}
To recall the ADMM, we focus on the convex minimization model with linear constraints and an objective function which is the sum of two functions without coupled variables:
\begin{equation}\label{Origianlproblem}
\begin{aligned}
\min_{x \in {X},y \in {Y}} \quad &\theta_1(x) + \theta_2(y) \\
s.t. \quad & Ax + By = b,
\end{aligned}
\end{equation}
where $\theta_1 : \mathbb{R}^{n_1} \rightarrow \mathbb{R}$ and $\theta_2 : \mathbb{R}^{n_2} \rightarrow \mathbb{R}$ are both convex (not necessarily smooth) functions, $A \in \mathbb{R}^{m \times n_1}$ and $B \in \mathbb{R}^{m \times n_2}$ are given matrices, $X \subseteq \mathbb{R}^{n_1}$ and $Y \subseteq \mathbb{R}^{n_2}$ are convex sets,  and $b\in\mathbb{R}^m$. Define the mapping $\phi : \mathbb{R}^{n_1} \times \mathbb{R}^{n_2} \times \mathbb{R}^{m} \rightrightarrows \mathbb{R}^{n_1} \times \mathbb{R}^{n_2} \times \mathbb{R}^{m}$ by
\begin{equation}\label{KKTmapping}
\phi(x,y,\lambda) = \begin{pmatrix}
\partial \theta_1(x) - A^T \lambda + \mathcal{N}_{{X}}(x)\\
\partial \theta_2(y) - B^T \lambda + \mathcal{N}_{{Y}}(y)\\
Ax + By - b
\end{pmatrix}.
\end{equation}
Then it is obvious that the KKT system can be written as  $0 \in \phi(x, y, \lambda)$, where $\lambda$ is the multiplier.

The iterative scheme of
generalized proximal version of the ADMM (GPADMM for short)  for (\ref{Origianlproblem}) reads as
{\small
	\begin{equation}\label{ADMM}
	\left\{
	\begin{aligned}
	x^{k+1} &= \arg \min_{x \in {X}} ~ \{\theta_1(x) - (\lambda^k)^T(Ax+By^k-b) + \frac{\beta}{2}\|Ax+By^k-b\|^2 + \frac{1}{2}\|x-x^k\|^2_{D_1}\}, \\
	y^{k+1} &= \arg \min_{y \in {Y}} ~ \{\theta_2(y) - (\lambda^k)^T(Ax^{k+1}+By-b) + \frac{\beta}{2}\| Ax^{k+1}+By-b\|^2  + \frac{1}{2}\|y-y^k\|^2_{D_2} \},\\
	\lambda^{k+1} & = \lambda^k - \beta( Ax^{k+1} + By^{k+1}-b),
	\end{aligned}
	\right.
	\end{equation}
where $\lambda^k$ is an estimate of the Lagrange multiplier, $\beta>0$ is a penalty parameter
{and $D_1,D_2$ are positive semidefinite matrices.}
By the optimality conditions for the subproblem in each iteration, we have
\begin{equation}\label{r_opt_con_k}
\begin{pmatrix}
D_1(x^{k}-x^{k+1}) - {\beta}A^TB(y^k - y^{k+1}) \\
D_2(y^{k}-y^{k+1}) \\
\frac{1}{\beta}(\lambda^{k}-\lambda^{k+1})
\end{pmatrix} \in 	\phi(x^{k+1},y^{k+1},\lambda^{k+1}).
\end{equation}
Following the perturbation technique, we introduce the perturbation as the difference between two consecutive generated points, i.e., $$p^k = (p^k_1, p^k_2, p^k_3) = \left(x^{k}-x^{k+1}, y^{k}-y^{k+1}, \lambda^{k}-\lambda^{k+1}\right) \in \mathbb{R}^{n_1} \times \mathbb{R}^{n_2} \times \mathbb{R}^{m}.$$ The approximate KKT condition (\ref{r_opt_con_k}) consequently results in the following inclusion in a more compact form
\begin{equation}\label{compactKKTmapping}
H\begin{bmatrix}
p_1^k \\ p_2^k \\ p_3^k
\end{bmatrix}
\in \phi(x^{k+1},y^{k+1},\lambda^{k+1}),
\end{equation}
with
\[
H = \begin{bmatrix}
D_1 & - {\beta}A^TB & 0 \\
0 & D_2&0 \\
0 & 0 &\frac{1}{\beta}I
\end{bmatrix}.
\]
The compact form (\ref{compactKKTmapping}) simply motivates the canonically perturbed KKT solution map $S : \mathbb{R}^{n_1} \times \mathbb{R}^{n_2} \times \mathbb{R}^{m} \rightrightarrows \mathbb{R}^{n_1} \times \mathbb{R}^{n_2} \times \mathbb{R}^{m} $
\begin{equation}\label{mappingS}
S(p) := \{ (x,y,\lambda) ~|~ p \in \phi(x,y,\lambda) \}
\end{equation}
where $p = (p_1, p_2, p_3) \in \mathbb{R}^{n_1} \times \mathbb{R}^{n_2} \times \mathbb{R}^{m}$ is regarded as the canonical perturbation. Given the convergence of PGADMM, the proof for its linear convergence rate toward a KKT solution $(x^{*},y^{*},\lambda^{*})$ is purely technical under the calmness of $S(p)$ at $(0, x^{*},y^{*},\lambda^{*})$, see, e.g., \cite{PEB-SIAMNUM,yang-han}.

If we focus on proximal version of the ADMM (PADMM for short), i.e., $D_2 = 0$ in the PGADMM \eqref{ADMM} (includes the original ADMM where $D_1 = D_2 = 0$), (\ref{compactKKTmapping}) reduces as
$$
H_0 p^k =
\begin{pmatrix}
D_1(x^{k}-x^{k+1}) - {\beta}A^TB(y^k - y^{k+1}) \\ 0 \\\frac{1}{\beta}(\lambda^{k}-\lambda^{k+1})
\end{pmatrix}
\in \phi(x^{k+1},y^{k+1},\lambda^{k+1}),
$$
where $$p^k = (p^k_1, p^k_2, p^k_3) = \left(x^{k}-x^{k+1}, y^{k}-y^{k+1}, \lambda^{k}-\lambda^{k+1}\right) \in \mathbb{R}^{n_1} \times \mathbb{R}^{n_2} \times \mathbb{R}^{m},$$
and
\[
H_0 = \begin{bmatrix}
D_1 & - {\beta}A^TB & 0 \\
0 & 0 &0 \\
0 & 0 &\frac{1}{\beta}I
\end{bmatrix}.
\]
In fact, the PADMM iteration introduces no perturbation to the KKT component $S_g$ where
$$
S_g := \{ (x,y,\lambda) ~|~ 0 \in \partial \theta_2(y) - B^T \lambda + \mathcal{N}_{{Y}}(y) \}.
$$
Inspired by this observation, in the recent paper \cite{PEB-SIAMNUM}, the perturbation technique motivates a partially perturbed KKT mapping $S_P : \mathbb{R}^{n_1} \times \mathbb{R}^{m} \rightrightarrows \mathbb{R}^{n_1} \times \mathbb{R}^{n_2} \times \mathbb{R}^{m} $ as
$$
S_P(p) := \{ (x,y,\lambda) \in S_g ~|~ p \in \phi_P(x,y,\lambda) \},
$$
where $\phi_P : \mathbb{R}^{n_1} \times \mathbb{R}^{n_2} \times \mathbb{R}^{m} \rightrightarrows \mathbb{R}^{n_1} \times \mathbb{R}^{m} $ is defined as
\begin{equation*}\label{phi-defi}
\phi_P(x,y,\lambda) = \begin{pmatrix}
\partial \theta_1(x) - A^T \lambda + \mathcal{N}_{{X}}(x)\\
Ax + By - b
\end{pmatrix}.
\end{equation*}
The calmness of $S_P$, which is specifically tailored to the sequence $\{(x^k,y^k,\lambda^k)\}$ generated by the PADMM, is in general weaker than the calmness of $S$. However, the calmness of $S_P$ suffices to ensure the linear rate convergence (see \cite{PEB-SIAMNUM} for more details).

\subsection{Error bound and linear convergence of PDHG}
We close this section by considering the min-max problem
\begin{equation}\label{problem}
\min_x\max_y ~ \phi(x,y):= \phi_1(x) + \langle y, \mathcal {K} x \rangle - \phi_2(y),
\end{equation}
where $x \in \mathbb{R}^n$, $y\in \mathbb{R}^m$, $\phi_1 : \mathbb{R}^n \rightarrow (-\infty, \infty]$ and $\phi_2 : \mathbb{R}^m \rightarrow (-\infty, \infty]$ are convex, proper, lower semicontinuous convex functions and $\mathcal {K} \in \mathbb{R}^{m \times n}$ is a coupling matrix.

The work \cite{CP} proposed a first-order primal-dual type method named primal-dual hybrid gradient (PDHG) method, where at any iteration both primal variable $x$
and dual variable $y$ are updated by descent and ascent gradient projection steps respectively. If the PDHG method in \cite{CP} is applied to the saddle-point problem (\ref{problem}) {with parameter $\theta = 1$}, the iteration scheme reads as the following.
$$\label{PDHG_eq}
\left\{ \begin{aligned}
x^{k+1} &= \arg\min_{x}\left\{ \phi_1(x) + \langle y^k,\mathcal {K} x \rangle + \frac{1}{2\tau}\|x - x^k\|^2 \right\}, \\		
y^{k+1} &= \arg\max_{y}\left\{  \langle y,\mathcal {K}(2x^{k+1}-x^k) \rangle - \phi_2(y)  - \frac{1}{2\sigma}\|y - y^k\|^2 \rangle \right\}, \\
\end{aligned} \right.
$$where $\tau, \sigma > 0$ are step size parameter.
At iteration $k$ of the PDHG, the optimality condition expresses as
$$
0 \in \begin{pmatrix}
\partial \phi_1(x^{k+1}) + \mathcal {K}^Ty^{k+1} \\
\partial \phi_2(y^{k+1}) - \mathcal {K}x^{k+1}
\end{pmatrix} + \begin{pmatrix}
\frac{1}{\tau}I & -\mathcal {K}^T \\
-\mathcal {K} & \frac{1}{\sigma}
\end{pmatrix}\begin{pmatrix}
x^{k+1} - x^k \\ y^{k+1} - y^k
\end{pmatrix}.
$$
Following the perturbation technique, we introduce perturbation to the place where the difference between two consecutive generated points appears, which further induces the canonically perturbed solution map $S : \mathbb{R}^{n} \times \mathbb{R}^{m} \rightrightarrows \mathbb{R}^{n} \times \mathbb{R}^{m} $
\begin{equation*}\label{mappingS_T}
S(p) := \{ (x,y) ~|~ p \in T(x,y) \}
\end{equation*}
where $p = (p_1, p_2) \in \mathbb{R}^{n} \times \mathbb{R}^{m}$ represents the canonical perturbation,
$T : \mathbb{R}^{n+ m}\rightrightarrows \mathbb{R}^{n+m}$ is a set-valued map defined as following
\begin{equation*}\label{TMdefi}
 T(x,y) := \begin{pmatrix}
\partial \phi_1(x) + \mathcal {K}^Ty \\
\partial \phi_2(y) - \mathcal {K}x
\end{pmatrix} .
\end{equation*}
Under the calmness of $S$, the linear convergence of PDHG is purely technical (see \cite{he2009proximal,he2012convergence} for details).

{\footnotesize{

}}

\section{Appendix}

\subsection{Proof of Lemma \ref{sufficient-cost-to-go}}
(1) Since $x^{k+1}$ is the optimal solution of the proximal operation (\ref{Basic_Proximal_Problem}) with $a=x^k-\gamma \nabla f(x^k)$, we have
\begin{eqnarray}
&& g(x^{k+1}) + \frac{1}{2\gamma}\left\| x^{k+1} - \left( x^k - \gamma \nabla f(x^k) \right) \right\|^2 \le g(x^k) + \frac{1}{2\gamma}\left\| \gamma \nabla f(x^k) \right\|^2,\nonumber
	\end{eqnarray}which can be reformulated as
	\begin{equation}\label{Sufficient-descent-inequality-1}
	g(x^{k+1}) + \frac{1}{2\gamma}\left\| x^{k+1} - x^k \right\|^2 + \left\langle \nabla f(x^k), x^{k+1} - x^k \right\rangle - g(x^k) \le 0.
	\end{equation}Furthermore, since $\nabla f(x)$ is globally Lipschitz continuous with the Lipschitz constant $L$, we have
	$$
	f(x^{k+1})\le f(x^k) + \left\langle \nabla f(x^k), x^{k+1} - x^k \right\rangle + \frac{L}{2}\left\| x^{k+1} - x^k \right\|^2.
	$$ Adding the above inequality to (\ref{Sufficient-descent-inequality-1}) we  obtain
	$$
	F(x^{k+1}) - F(x^k) \le \left( \frac{L}{2} - \frac{1}{2\gamma} \right) \left\| x^{k+1} - x^k \right\|^2.
	$$As a result if $\gamma < \frac{1}{L}$ we have (\ref{Sufficient-descent-inequality-0}) with $\kappa_1 := \frac{1}{2\gamma} - \frac{L}{2}$.
	
	\medskip
	(2) By the optimality of  $x^{k+1}$ we have that for any $x$,
	$$
	g(x^{k+1}) + \frac{1}{2\gamma}\left\| x^{k+1} - x^k + \gamma \nabla f(x^k) \right\|^2 \le g(x) + \frac{1}{2\gamma}\left\| x - x^k + \gamma \nabla f(x^k) \right\|^2,
	$$ which can be reformulated as
	$$
	g(x^{k+1})  - g(x) \le \frac{1}{2\gamma}\left\| x - x^k \right\|^2 - \frac{1}{2\gamma}\left\| x^{k+1} - x^k \right\|^2 + \left\langle \nabla f(x^k), x - x^{k+1}  \right\rangle.
	$$
	By the Lipschitz continuity of  $\nabla f(x)$,
	$$
	f(x) \ge f(x^{k+1}) + \left\langle \nabla f(x^{k+1}), x - x^{k+1} \right\rangle - \frac{L}{2}\left\| x - x^{k+1} \right\|^2.
	$$  By  the above two  inequalities we obtain
	\begin{eqnarray}
	F(x^{k+1}) - F(x)&\le&\frac{1}{2\gamma}\left\| x - x^k \right\|^2 - \frac{1}{2\gamma}\left\| x^{k+1} - x^k \right\|^2 + \left\langle \nabla f(x^k), x - x^{k+1} \right\rangle \nonumber\\
	&&\qquad - \left\langle \nabla f(x^{k+1}), x - x^{k+1} \right\rangle + \frac{L}{2} \left\| x - x^{k+1} \right\|^2\nonumber\\
	&\le&\frac{1}{\gamma}\left\| x - x^{k+1}\right\|^2 + \frac{1}{\gamma} \left\| x^{k+1} - x^k \right\|^2 - \frac{1}{2\gamma}\left\| x^{k+1} - x^k \right\|^2 + \nonumber\\
	&&\qquad \left\langle \nabla f(x^k) - \nabla f(x^{k+1}), x - x^{k+1} \right\rangle + \frac{L}{2} \left\| x - x^{k+1} \right\|^2\nonumber\\
	&\le&\frac{1}{\gamma}\left\| x - x^{k+1}\right\|^2+ \frac{1}{\gamma} \left\| x^{k+1} - x^k \right\|^2-\frac{1}{2\gamma} \left\| x^{k+1} - x^k \right\|^2 + \frac{L}{2} \left\| x^{k+1} - x^k \right\|^2 + \nonumber\\
	&&\qquad  \frac{1}{2} \left\| x - x^{k+1} \right\|^2 + \frac{L}{2} \left\| x - x^{k+1} \right\|^2\nonumber\\
	&=& \left( \frac{1}{\gamma} + \frac{L+1}{2} \right) \left\| x - x^{k+1} \right\|^2 + \left( \frac{L}{2} + \frac{1}{2\gamma} \right) \left\| x^k - x^{k+1} \right\|^2,
	\end{eqnarray}
	from which we can obtain (\ref{cost-to-go-estimate-0}) with $\kappa_2 := \max\left\{ \left( \frac{1}{\gamma} + \frac{L+1}{2} \right), \left( \frac{L}{2} + \frac{1}{2\gamma} \right) \right\}$.

\subsection{Proof of Theorem \ref{Thm3.2}}
	In the proof, we denote by $\zeta:=F(\bar x)$ for succinctness. And we recall that the proper separation of the stationary value condition holds on $\bar{x} \in {{\cal X}}^\pi$, i.e.,
	there exists $ \delta > 0$ such that
	\begin{equation}\label{PBisocost_proof}
	x \in {{\cal X}}^\pi\cap {\mathbb{B}} ( \bar{x},\delta )\quad \Longrightarrow \quad F(x) = F(\bar{x}).
	\end{equation}
	Without lost of generality, we assume that {$\epsilon < \delta/(\kappa +1)$} throughout the proof.
	
	Step 1. We prove that $\bar x$ is a stationary point and
	\begin{equation} \lim_{k\rightarrow \infty} \|x^{k+1}
	-x^{k}\|=0.\label{equalconv}  \end{equation}
	Adding the inequalities in (\ref{Sufficient-descent-inequality-0}) starting from iteration $k=0$ to an arbitrary positive integer $K$, we obtain
	$$
	\sum_{k=0}^{K} \left\| x^{k+1} - x^k \right\|^2 \le \frac{1}{\kappa_1} \left( F(x^0) - F(x^{K+1}) \right)\le \frac{1}{\kappa_1} \left( F(x^0) - F_{\min} \right) < \infty.
	$$ It follows that
	$
	\sum_{k=0}^{\infty} \left\| x^{k+1} - x^k \right\|^2 < \infty,
	$ and consequently
	(\ref{equalconv}) holds.  Let $\{ x^{k_i} \}_{i=1}^\infty$ be a  convergent subsequence of $\left\{ x^k \right\}$ such that $x^{k_i}\rightarrow \bar{x}$ as $i\rightarrow \infty$. Then by (\ref{equalconv}), we have
	\begin{equation}
	\lim_{i\rightarrow \infty}  x^{k_i}=\lim_{i\rightarrow \infty} x^{k_i-1}=\bar x.\label{eqn1}
	\end{equation}
	Since
	\begin{equation}\label{x-bar-stationary}
	x^{k_i}
	\in\hbox{Prox}_g^\gamma \left  (x^{k_i-1} -\gamma \nabla f(x^{k_i-1}) \right ),
	\end{equation}
	let $i\rightarrow \infty$ in (\ref{x-bar-stationary}) and by the outer semicontinuity of $\hbox{Prox}_g^\gamma (\cdot)$ (see \cite[Theorem 1.25]{Rockafellar2009variational}) and continuity of $\nabla f$, we have
	\[
	\bar{x}\in \hbox{Prox}_g^\gamma \left(\bar{x} -\gamma \nabla f(\bar{x}) \right ),
	\]
Using the definition of the proximal operator and applying the optimality condition and we have
	$$
	0\in  \nabla f \left(\bar{x}\right) + \partial^\pi g \left(\bar{x}\right) ,
	$$
	and so $\bar x\in {\cal X}^\pi$.
	
	Step 2. Given $\hat{\epsilon} > 0$ such that $\hat{\epsilon} <  \delta/\epsilon - \kappa -1$, for each $k > 0$, we can find $\bar{x}^{k} \in {\tilde{\cal{X}}}$ such that  $$\left\| \bar{x}^{k} - x^{k} \right\| \le \min\left\{\sqrt{d\left( x^{k},\tilde{\cal{X}} \right)^2 + \hat{\epsilon}\|x^k - x^{k-1}\|^2}, d\left( x^{k},\tilde{\cal{X}} \right) + \hat{\epsilon}\|x^k - x^{k-1}\|\right\}.$$ It follows by  the cost-to-estimate condition (\ref{cost-to-go-estimate-0}) we have
	\begin{equation}\label{cost-to-go-estimate-01}
	F(x^{k}) - F(\bar{x}^{k})\le \hat{ \kappa}_2 \left( \hbox{dist} \left( x^{k}, {\cal{X}}^\pi \right)^2 + \left\| x^{k} - x^{k-1} \right\|^2 \right),
	\end{equation}
	with $\hat{ \kappa}_2 = \kappa_2(1+\hat{\epsilon})$.
	Now we use the method of mathematical induction to prove that there exists $k_{\ell}>0$ such that
	for all $j \geq  k_{\ell}$,
	\begin{eqnarray}\label{neighborhood-0}
	&&	x^{j}\in \mathbb{B} \left( \bar{x}, \epsilon \right),\quad x^{j+1}\in \mathbb{B} \left( \bar{x}, \epsilon \right),\quad F(\bar{x}^j) = \zeta ,\quad {{F(\bar{x}^{j+1}) = \zeta}},\\
	&&\label{cost-to-go-estimate-1}
	F(x^{j+1}) - \zeta \le \hat{ \kappa}_2 \left( \hbox{dist} \left( x^{j+1}, {\cal{X}}^\pi \right)^2 + \left\| x^{j+1} - x^j \right\|^2 \right),
	\end{eqnarray}
	\begin{equation}
	\label{sum-iteration-error-0}
	\sum\limits_{i=k_{\ell}}^j \left\| x^i - x^{i+1} \right\| \le \frac{\left\| x^{k_{\ell}-1} - x^{k_{\ell}} \right\| - \left\| x^{j} - x^{j+1} \right\|}{2} + c\left[\sqrt{F(x^{k_{\ell}}) - \zeta} - \sqrt{F(x^{j+1}) - \zeta}\right],
	\end{equation} where the  constant $c:=\frac{2\sqrt{\hat{ \kappa}_2(\kappa^2+1)}}{\kappa_1}> 0$.
	
	By (\ref{eqn1})  and the fact that $F$ is continuous in its domain,
	there exists  $k_{\ell} > 0$ such that $x^{k_\ell}\in \mathbb{B} \left( \bar{x},\epsilon \right)$, $x^{k_\ell+1}\in \mathbb{B} \left( \bar{x},\epsilon \right)$,
	\begin{eqnarray}
	&&	{{\left\| x^{k_{\ell}} - \bar{x} \right\| + \frac{\left\| x^{k_{\ell}-1} - x^{k_{\ell}} \right\| }{2} + c\left[\sqrt{F(x^{k_{\ell}}) - \zeta}\right]\le \frac{\epsilon}{2}}},
	\label{starstar} \\
	&&
	{{\left\| x^{k+1} - x^{k} \right\| < \frac{\epsilon}{2} , \quad \forall k \ge k_{\ell} - 1}},
	\label{star}
	\end{eqnarray}	
	$$
	{{\left\|\bar{x}^{k_{\ell}} - \bar{x}\right\| \le \left\|\bar{x}^{k_{\ell}} - x^{k_{\ell}}\right\| + \left\|x^{k_{\ell}} - \bar{x}\right\| \overset{(\ref{error-bound-condition-1})}{\le} (\kappa+\hat{\epsilon})  \left\|x^{k_{\ell}} - x^{k_{\ell} - 1}\right\| + \left\|x^{k_{\ell}} - \bar{x}\right\| < {(\kappa +\hat{\epsilon} + 2)\epsilon/2} < \delta}},
	$${\color{black}{which indicates $\bar{x}^{k_{\ell}}\in \tilde{\cal{X}} \cap {\mathbb{B}}\left( \bar{x},\delta \right)$.
			It follows by the proper separation of the stationary value condition (\ref{PBisocost_proof}) that  $F\left( \bar{x}^{k_{\ell}} \right) = \zeta$.}}
	
	Before inducing (\ref{neighborhood-0})-(\ref{sum-iteration-error-0}), we should get ready by showing that for $j\ge k_{\ell}$,
	{if (\ref{neighborhood-0}) and (\ref{cost-to-go-estimate-1}) hold}, then
	\begin{equation}\label{iteration-error-bound}
	2\left\| x^{j} - x^{j+1} \right\| \le c\left[\sqrt{F(x^{j}) - \zeta} - \sqrt{F(x^{j+1}) - \zeta}\right] + \frac{\left\| x^{j} - x^{j+1} \right\| + \left\| x^{j-1} - x^{j} \right\|}{2}.
	\end{equation}
	Firstly, since $x^{j}\in \mathbb{B} \left( \bar{x}, \epsilon \right)$, $F(\bar{x}^j) = \zeta$ and (\ref{cost-to-go-estimate-01}) holds,
	it follows from (\ref{error-bound-condition-1}) that \begin{equation}
	F(x^j)-\zeta\leq \hat{ \kappa}_2(\kappa \|x^j-x^{j-1}\|^2+ \|x^j-x^{j-1}\|^2) =\kappa_3^2\|x^j-x^{j-1}\|^2,\label{new}
	\end{equation}
	where  $\kappa_3 := \sqrt{\hat{ \kappa}_2 \left( \kappa^2 + 1 \right)}.$
	Similarly, since $x^{j+1}\in \mathbb{B} \left( \bar{x}, \epsilon \right)$ and {{$F(\bar{x}^{j+1}) = \zeta$, by (\ref{cost-to-go-estimate-01})}} and  condition (\ref{error-bound-condition-1}), we have
	{\begin{eqnarray}\label{cost-to-go-estimate-2}
		F(x^{j+1}) - \zeta
		&\le&\kappa_3^2 \left\| x^{j+1} - x^j \right\|^2.
		\end{eqnarray}} As a result, we can obtain
	\begin{eqnarray}
	\sqrt{F(x^{j}) - \zeta} - \sqrt{F(x^{j+1}) - \zeta}&=&\frac{\left(F(x^j) - \zeta\right) - \left(F(x^{j+1}) - \zeta\right)}{\sqrt{F(x^{j}) - \zeta} + \sqrt{F(x^{j+1}) - \zeta}}\nonumber\\
	&=& \frac{F(x^j) -F(x^{j+1})}{\sqrt{F(x^{j}) - \zeta} + \sqrt{F(x^{j+1}) - \zeta}}\nonumber\\
	&\overset{(\ref{Sufficient-descent-inequality-0}) (\ref{new})(\ref{cost-to-go-estimate-2})}{\ge} & \frac{\kappa_1 \left\| x^{j+1} - x^j \right\|^2}{\kappa_3\left( \left\| x^{j} - x^{j-1} \right\| + \left\| x^{j+1} - x^j \right\| \right)}.
	\end{eqnarray}After defining $c: = \frac{2\kappa_3}{\kappa_1}$, we have
	$$
	\left( c \left[\sqrt{F(x^{j}) - \zeta} - \sqrt{F(x^{j+1}) - \zeta}\right] \right)\left( \frac{\left\| x^{j} - x^{j+1} \right\| + \left\| x^{j-1} - x^{j} \right\|}{2} \right)\ge \left\| x^{j+1} - x^j \right\|^2,
	$$ from which by applying $ab\le \left( \frac{a+b}{2} \right)^2$ we establish (\ref{iteration-error-bound}).

	Next we proceed to prove the three properties (\ref{neighborhood-0})-(\ref{sum-iteration-error-0}) by induction on $j$. For $j = k_{\ell}$, we have
	$$
	x^{k_{\ell}}\in \mathbb{B} \left( \bar{x}, \epsilon \right),\quad x^{k_{\ell}+1}\in \mathbb{B} \left( \bar{x}, \epsilon \right),\quad F(\bar{x}^{k_\ell}) = \zeta,
	$$
	and since
		$$
	\begin{aligned}
	\left\| \bar{x}^{k_{\ell}+1} - \bar{x} \right\| &\le \| \bar{x}^{k_{\ell}+1} - x^{k_{\ell}+1} \| + \| x^{k_{\ell}+1} -\bar{x} \| \\
	&\le (\kappa+\hat{\epsilon})\| x^{k_{\ell}+1} - x^{k_{\ell}} \| + \epsilon \|\\
	& < (\kappa + \hat{\epsilon} +2)\epsilon/2 < \delta,
	\end{aligned}
	$$
	where the second inequality follows from the definition of $\bar{x}^k$ and \eqref{error-bound-condition-1}, the third inequality follows from (\ref{star}). {\color{black}{It follows by (\ref{PBisocost_proof}) that  $F(\bar{x}^{k_{\ell}+1}) = \zeta$}}, and hence by \eqref{cost-to-go-estimate-01},
	$$
	F(x^{k_{\ell}+1}) - \zeta \le \hat{ \kappa}_2 \left( \hbox{dist} \left( x^{k_{\ell}+1}, \tilde{\cal{X}} \right)^2 + \left\| x^{k_{\ell}+1} - x^{k_\ell} \right\|^2 \right),
	$$
	which is (\ref{cost-to-go-estimate-1}) with $j=k_{\ell}$.
	Note that property	(\ref{sum-iteration-error-0}) for $j = k_{\ell}$ can  be obtained directly through (\ref{iteration-error-bound}).
	
	Now  suppose  (\ref{neighborhood-0}) \eqref{cost-to-go-estimate-1} and (\ref{sum-iteration-error-0}) hold for certain $j> k_{\ell}$. By induction we also want to show that  (\ref{neighborhood-0}) \eqref{cost-to-go-estimate-1} and (\ref{sum-iteration-error-0}) hold for $j+1$.  We have
	\begin{eqnarray}
	\left\| x^{j+2} - \bar{x} \right\|&\le& \left\| x^{k_{\ell}} - \bar{x} \right\| + \sum_{i={k_{\ell}}}^{j} \left\| x^i - x^{i+1} \right\| + \left\| x^{j+1} - x^{j+2} \right\| \nonumber\\
	&<& \left\| x^{k_{\ell}} - \bar{x} \right\| + \frac{\left\| x^{k_{\ell}-1} - x^{k_{\ell}} \right\| - \left\| x^{j} - x^{j+1} \right\|}{2} + c\left[\sqrt{F(x^{k_{\ell}}) - \zeta} - \sqrt{F(x^{j+1}) - \zeta}\right] + \frac{\epsilon}{2} \nonumber\\
	&\le& \left\| x^{k_{\ell}} - \bar{x} \right\| + \frac{\left\| x^{k_{\ell}-1} - x^{k_{\ell}} \right\| }{2} + c\left[\sqrt{F(x^{k_{\ell}}) - \zeta}\right] + \frac{\epsilon}{2} \le \epsilon,\nonumber
	\end{eqnarray}
	where the second inequality follows from (\ref{sum-iteration-error-0}) and (\ref{star}) and the last inequality follows from (\ref{starstar}).
	{Since  $x^{j+2}\in \mathbb{B}(\bar x, \epsilon)$}, by the definition of $\bar{x}^j$ and (\ref{error-bound-condition-1}), there holds that
	$$\begin{aligned}
	\left\| \bar{x}^{j+2} - \bar{x} \right\| &\le \| \bar{x}^{j+2} - x^{j+2} \| + \| x^{j+2} - \bar{x} \| \\
	&\le (\kappa+\hat{\epsilon})\| x^{j+2} - x^{j+1} \| + \epsilon \\
	& < (\kappa +\hat{\epsilon} +2)\epsilon/2 <  \delta,
	\end{aligned}$$
	where the third inequality follows from (\ref{star}). {\color{black}{It follows from the proper separation of stationary value assumption (\ref{PBisocost_proof}) that $F(\bar{x}^{j+2}) = \zeta$. }} Consequently by \eqref{cost-to-go-estimate-01}, we have
	\begin{equation*}
	F(x^{j+2}) - \zeta \le \hat{ \kappa}_2 \left( \hbox{dist} \left( x^{j+2}, \tilde{\cal{X}} \right)^2 + \left\| x^{j+2} - x^{j+1} \right\|^2 \right).
	\end{equation*}
	So far we have shown that (\ref{neighborhood-0})-(\ref{cost-to-go-estimate-1}) hold for $j+1$.
	Moreover
	\begin{eqnarray}
	&&\hspace{-25pt} \sum\limits_{i=k_{\ell}}^{j+1} \left\| x^i - x^{i+1} \right\| \nonumber\\
	&\overset{(\ref{sum-iteration-error-0})}{\le}& \frac{\left\| x^{k_{\ell}-1} - x^{k_{\ell}} \right\| - \left\| x^{j} - x^{j+1} \right\|}{2} + c\left[\sqrt{F(x^{k_{\ell}}) - \zeta} - \sqrt{F(x^{j+1}) - \zeta}\right] + \left\| x^{j+1} - x^{j+2} \right\|\nonumber\\
	&\overset{(\ref{iteration-error-bound}) \mbox{ for } j+1}{\le}& \frac{\left\| x^{k_{\ell}-1} - x^{k_{\ell}} \right\| - \left\| x^{j} - x^{j+1} \right\|}{2} + c\left[\sqrt{F(x^{k_{\ell}}) - \zeta} - \sqrt{F(x^{j+1}) - \zeta}\right] + \nonumber\\
	&& c\left[\sqrt{F(x^{j+1}) - \zeta} - \sqrt{F(x^{j+2}) - \zeta}\right] + \frac{\left\| x^{j+1} - x^{j+2} \right\| + \left\| x^{j} - x^{j+1} \right\|}{2} - \left\| x^{j+1} - x^{j+2} \right\|\nonumber\\
	&=& \frac{\left\| x^{k_{\ell}-1} - x^{k_{\ell}} \right\| - \left\| x^{j+1} - x^{j+2} \right\|}{2} + c\left[\sqrt{F(x^{k_{\ell}}) - \zeta} - \sqrt{F(x^{j+2}) - \zeta}\right],\nonumber
	\end{eqnarray} from which we obtain  (\ref{sum-iteration-error-0}) for $j+1$.   The desired induction on $j$ is now complete.  In summary, we have now proved the properties (\ref{neighborhood-0})-(\ref{sum-iteration-error-0}).

	Step 3. We prove that the whole sequence $\{x^k\}$ converges to $\bar x$ and  (\ref{linear-convergence-objective})-(\ref{linear-convergence-sequence}) hold.
	
	By (\ref{sum-iteration-error-0}),
	for all $j\ge k_{\ell}$
	\begin{eqnarray*}
		\sum\limits_{i=k_{\ell}}^j \left\| x^i - x^{i+1} \right\| &\le& \frac{\left\| x^{k_{\ell}-1} - x^{k_{\ell}} \right\| - \left\| x^{j} - x^{j+1} \right\|}{2} + c\left[\sqrt{F(x^{k_{\ell}}) - \zeta} - \sqrt{F(x^{j+1}) - \zeta}\right]\nonumber\\
		&\le& \frac{\left\| x^{k_{\ell}-1} - x^{k_{\ell}} \right\|}{2} + c \sqrt{F(x^{k_{\ell}}) - \zeta} < \infty,
	\end{eqnarray*}which indicates that $\left\{ x^k \right\}$ is a Cauchy sequence. It follows that the whole sequence converges to
	the stationary point
	$\bar{x}$.
	{{Further for all $k \ge k_{\ell}$, we have $x^k\in \mathbb{B}(\bar{x},\epsilon)$}}. As a result, the PG-iteration-based error bound condition (\ref{error-bound-condition-1}) holds on all the iteration points $\left\{ x^k \right\}_{k > k_{\ell}}$. Recall that by (\ref{Sufficient-descent-inequality-0}) and (\ref{cost-to-go-estimate-2}), we have
	\begin{eqnarray}
	F(x^{k+1}) - F(x^k) &\le & -\kappa_1\left\| x^{k+1} - x^k \right\|^2, \nonumber \\
	F(x^{k+1}) - \zeta &\le & \hat{ \kappa}_2 \left( \kappa^2 + 1 \right)\left\| x^{k+1} - x^k \right\|^2 ,\nonumber
	\end{eqnarray}which implies that
	$$
	F(x^k) - F(x^{k+1})\ge \frac{\kappa_1}{\hat{ \kappa}_2\left( \kappa^2 + 1 \right)} \left(F(x^{k+1}) - \zeta\right).
	$$
	We can observe easily that
	$$
	F(x^k) - \zeta + \zeta - F(x^{k+1})\ge \frac{\kappa_1}{\hat{ \kappa}_2\left( \kappa^2 + 1 \right)} \left(F(x^{k+1}) - \zeta\right).
	$$
	Thus we have
	$$
	F(x^{k+1}) - \zeta \le \sigma \left(F(x^{k}) - \zeta\right), \mbox{with } \sigma := \frac{1}{1 + \frac{\kappa_1}{\hat{ \kappa}_2\left( \kappa^2 + 1 \right)}} < 1,
	$$
	which completes the proof of  (\ref{linear-convergence-objective}).
	
	In addition, we have shown that for any $j \ge k_{\ell}$, we have $x^{j}\in \mathbb{B} \left( \bar{x}, \epsilon \right)$, $x^{j+1} \in \mathbb{B} \left( \bar{x}, \epsilon \right)$, $F(\bar{x}^j) = \zeta$, $F(\bar{x}^{j+1}) = \zeta$. And by \eqref{iteration-error-bound},
	\begin{equation*}
	\left\| x^j - x^{j+1} \right\| \le c \left[\sqrt{F(x^{j}) - \zeta} - \sqrt{F(x^{j+1}) - \zeta}\right] + \frac{\left\| x^{j-1} - x^{j} \right\| - \left\| x^{j} - x^{j+1} \right\|}{2}.
	\end{equation*}
	For any $K \ge k \ge k_{\ell}$, by summing the above inequality from $j = k$ to $j = K$, it follows that
	\begin{equation*}
	\sum\limits_{j=k}^K \left\| x^j - x^{j+1} \right\| \le \frac{\left\| x^{k-1} - x^{k} \right\|}{2} + c \sqrt{F(x^{k}) - \zeta}.
	\end{equation*} Denote by $
	\Delta_k := \sum_{j=k}^{\infty} \left\| x^j - x^{j+1} \right\|.
	$ By the above inequality, we have that for any $k \ge k_{\ell}$,
	\begin{eqnarray}
	\Delta_{k} &\le& \frac{\left\| x^{k-1} - x^{k} \right\|}{2} + c \sqrt{F(x^{k}) - \zeta} \nonumber \\
	& \overset{(\ref{new})}{\le}& \frac{\left\| x^{k-1} - x^{k} \right\|}{2} + c \sqrt{\hat{ \kappa}_2 \left( \kappa^2 + 1 \right)}\left\| x^{k-1} - x^{k} \right\| \nonumber \\
	& = & \left(\frac{1}{2} + c \sqrt{\hat{ \kappa}_2 \left( \kappa^2 + 1 \right)} \right)(\Delta_{k-1}-\Delta_{k}).\nonumber
	\end{eqnarray}Therefore we obtain
	\begin{equation*}
	\Delta_k \le \rho \Delta_{k-1},\quad \forall k\ge k_{\ell},\quad \rho: = 1 - \frac{2}{3+2 c \sqrt{\hat{ \kappa}_2 \left( \kappa^2 + 1 \right)}} < 1.
	\end{equation*}
	Since  $x^k \rightarrow \bar{x}$, it follows that  $\left\| x^k - \bar{x} \right\| \le \sum_{i=k}^{\infty} \left\| x^i - x^{i+1} \right\| = \Delta_k$.
	Consequently,
	\begin{equation*}
	\|x^k - \bar{x}\| \le \rho^{k-k_{\ell}}  \Delta_{k_{\ell}} = \rho_0\rho^k,\quad \forall k\ge k_{\ell}, \quad \rho_0 = \frac{\Delta_{k_{\ell}}}{\rho^{k_{\ell}}},
	\end{equation*}
	which proves (\ref{linear-convergence-sequence}).

\end{document}